\documentclass[11pt,a4paper]{amsart}
\usepackage{pdfsync}
\usepackage{mathptmx} 
\usepackage[scaled=0.90]{helvet} 
\usepackage{courier} 
\normalfont
\usepackage[T1]{fontenc}

\usepackage{amssymb}
\usepackage{graphicx}
\usepackage{epstopdf}
\usepackage{lscape}

\renewcommand{\epsilon}{\varepsilon}
\renewcommand{\setminus}{\smallsetminus}
\renewcommand{\emptyset}{\varnothing}

\newtheorem{proposition}[subsection]{Proposition}
\newtheorem{corollary}[subsection]{Corollary}
\newtheorem{lemma}[subsection]{Lemma}

\newtheorem{example}[subsection]{Example}
\newtheorem*{theorema}{Theorem A}
\newtheorem*{theoremb}{Theorem B}
\newtheorem*{theoremb1}{Theorem B1}
\newtheorem*{theoremb2}{Theorem B2}
\newtheorem*{conjectureb}{Conjecture B}

\theoremstyle{definition}

\newtheorem{definition}[subsection]{Definition}
\newtheorem{notation}[subsection]{Notation}

\theoremstyle{remark}
\newtheorem{remark}[subsection]{Remark}

\newcommand{\N}{\mathbb N}
\newcommand{\R}{\mathbb R}

\newcommand{\RP}{\mathbb R\mathbb P}

\newcommand{\FP}{\operatorname{FP}}
\newcommand{\hull}{\operatorname{hull}}

\title[Tame subsets of spheres]
{ On subsets of $S^n$ whose $(n+1)$-point subsets are contained in open hemisheres }
\author{Robert Bieri}
\author{Peter Kropholler}
\author{Brendan Owens}

\address{\newline Institut f\"ur Mathematik,
Goethe-Universit\"at,
Robert-Mayer-Str. 6-10 60325 Frankfurt, Germany
\newline Department of Mathematical Sciences, Binghamton University, SUNY, Binghamton, NY 13902-6000 USA}
\email{rbieri@math.binghamton.edu}
\address{\newline Department of Mathematics,
310 Malott Hall,
Cornell University,
Ithaca, NY 14853-4201 USA
\newline School of Mathematics and Statistics, University of Glasgow, 15 University Gardens, Glasgow G12 8QW, United Kingdom}
\email{peter.kropholler@glasgow.ac.uk}
\address{\newline School of Mathematics and Statistics, University of Glasgow, 15 University Gardens, Glasgow G12 8QW, United Kingdom}
\email{brendan.owens@glasgow.ac.uk}
\thanks{
The second author was supported in part by the Mittag-Leffler Institute, Djursholm, and also wishes to thank the Goethe-Universit\"at Frankfurt for hospitality during the preparation of this manuscript. 
The third author was supported in part by EPSRC grant EP/I033754/1.}

\date{\today} 

\begin{document}

\begin{abstract}
We investigate the nature of subsets of spheres which satisfy a tameness condition associated with the Bieri--Groves $\FP_{m}$-conjecture. We find that there is a natural polyhedrality in the case of $n$-tame subsets of an $(n-1)$-sphere. In the case $n=3$ we establish a strong polyhedrality condition for certain maximal open $3$-tame sets. Many examples are included.
\end{abstract}


\maketitle

\tableofcontents

\section{Statement of results; notion of maxtame set}

Let $S^{n-1}$ denote the unit sphere in Euclidean $n$-space $\R^n$.
For a natural number $m$ we use the terminology \emph{$m$-point subset} to mean a non-empty subset of \emph{at most} $m$ points. We shall say that a finite subset of $\R^{n}$ is \emph{balanced} if and only if a choice of point masses can be assigned to its elements so that its centre of mass is at the origin, or equivalently if its convex hull contains the origin.
A subset $V$ of $S^{n-1}$ is said to be \emph{$m$-tame} if and only if every $m$-point subset of $V$ is contained in some open hemisphere of $S^{n-1}$, or equivalently if $V$ has no balanced $m$-point subsets.
The notion of $m$-tameness arose in the context of studying cohomological finiteness conditions for metabelian groups, see especially joint work of the first author with Groves:\ \cite{bierigroves1982,bierigroves1984,bierigroves1985,bierigroves1986}. In particular it was shown that the $\Sigma^{c}$ invariants associated with metabelian groups are closed, rationally defined and polyhedral. In this article we draw attention to a manifestation of polyhedrality in the abstract setting of tame subsets of spheres. Rather than considering closed sets, it is convenient to consider open sets. 

It is well known that every closed $(n+1)$-tame subset of $S^{n-1}$ is contained in an open hemisphere. The same is true of open $(n+1)$-tame subsets. These two facts are covered by Lemma \ref{completelytamecase}. As a consequence, the $m$-tameness condition is most interesting to study when $2\le m\le n$. In this article we prove two theorems on polyhedrality when $m=n$ and illustrate them with examples and counterexamples in the case $n=3$ (and $m=2$ or $3$). We use the following terminology.

\begin{definition}
A subset $U$ of $S^{n-1}$ is called \emph{weakly maxtame} if $U$ is open and $n$-tame, $\pi_{0}(U)$ is finite, and $U$ is maximal amongst open $n$-tame sets with at most $|\pi_{0}(U)|$ components. A weakly maxtame set is called \emph{maxtame} if it is maximal amongst all open $n$-tame subsets of $S^{n-1}$.
\end{definition}

Our first theorem shows that polyhedrality is an intrinsic property of weakly maxtame sets.

\begin{theorema}
Let $U$ be a weakly maxtame subset of $S^{n-1}$. Then there is a finite family $\mathcal H$ of open hemispheres of $S^{n-1}$ such that
\begin{enumerate}
\item
each component of $U$ is an intersection of members of $\mathcal H$; and
\item
the closure of each $H\in \mathcal H$ has a neighbourhood $N$ such that $U\cap H = U\cap N$.
\end{enumerate} 
In particular, distinct components of $U$ have disjoint closures.
\end{theorema}

We have the following conjecture which motivates the study of maxtame and weakly maxtame sets.

\begin{conjectureb}
Let $U$ be an $n$-tame subset of $S^{n-1}$ which is either
\begin{enumerate}
\item
closed; or
\item
open with finitely many connected components.
\end{enumerate}
Then $U$ is contained in a maxtame subset.
\end{conjectureb}

The conjecture is elementary if $n\le2$.
We establish the conjecture for $n=3$ (i.e. for $3$-tame sets on a $2$-sphere) and we also prove a weak version of the conjecture for arbitrary $n$.

\begin{theoremb}
Let $U$ be as in Conjecture B. Then
\begin{itemize}
\item
[\textbf{B1}] $U$ is contained in a weakly maxtame subset; and
\item
[\textbf{B2}] if $n\le 3$, $U$ is contained in a maxtame subset.
\end{itemize}
\end{theoremb}

We go on to exhibit examples of maxtame sets on $S^{2}$. We shall also see that maximal open $2$-tame subsets of $S^{2}$ need not be polyhedral even if they have a finite number of components: in fact there are non-polyhedral examples having a single component or any other number of components. 

When $n=2$ a complete classification and description of (weakly) maxtame sets is easy to obtain and we also go on to describe their configuration space.

\section{Embedding closed $n$-tame sets into polyhedral $n$-tame sets}

In this section we shall prove the first part of Theorem B.

\begin{theoremb1}
Let $U$ be as in the conjecture.
Then $U$ is contained in a weakly maxtame subset.
\end{theoremb1}

Let $\mathcal T(\ell,m,n)$ denote the set of open $m$-tame subsets of $S^{n-1}$ having at most $\ell$ components, partially ordered by inclusion. Using Zorn's lemma it is immediate that $\mathcal T(\ell, m, n)$ has maximal elements. If $U$ is any open $m$-tame set then $U$ belongs to $\mathcal T(|\pi_{0}(U)|,m,n)$ and if $|\pi_{0}(U)|$ is finite then the maximal elements here are weakly maxtame sets containing $U$.  This proves  Theorem B1 in case (ii).

If $U$ is closed and $m$-tame we shall show that $U$ has an open $m$-tame neighbourhood $U'$. By compactness, the union $U''$ of some finite number of components of $U'$ covers $U$ and we may choose a weakly maxtame set containing $U''$ as above.

Therefore, to complete the proof of Theorem B1 it suffices to show that every closed $m$-tame subset is contained in an open $m$-tame subset. The following lemma is used in the proof. The conclusion we reach, Lemma \ref{sec2:lemma3}(iii) below, is stronger than needed here but is used in its stronger form to prove the Key Lemma \ref{keylemma} below which plays a part in the proof of Theorem A.

\begin{lemma}\label{sec2:lemma2}
Let $A_{1},\dots,A_{m}$ be closed subspaces of a compact metrizable space $X$. Suppose that $B$ is a closed subspace of $\underbrace{X\times\dots\times X}_{m}$ disjoint from $A_{1}\times\dots\times A_{m}$. Then there exist open sets $U_{i}\supseteq A_{i}$ and $V\supseteq B$ such that $V$ is disjoint from $U_{1}\times\dots\times U_{m}$. Moreoever if all the $A_{i}$ are equal then we may choose all $U_{i}$ equal.
\end{lemma}
\begin{proof}
Since $X$ is compact and Hausdorff we can separate $A_{1}\times\dots\times A_{m}$ and $B$ by a pair of open sets $U$ and $V$, so it suffices to show that any neighbourhood $U$ of $A$ contains an open neighbourhood of the form $U_{1}\times\dots\times U_{m}$ where $A_{i}\subseteq U_{i}\subseteq X$. Suppose for a contradiction that this is not the case and let $U$ be an offending open neighbourhood of $A$. Let $d$ be a metric on $X$ matching its topology. For each natural number $n$ let $W_{i,n}:=\{(x\in X;\ d(x,A_{i})<\frac1n\}$, and then choose an $m$-tuple $v_{n}\in W_{1,n}\times\dots\times W_{m,n}\setminus U$ witnessing the assumption that $U$ is a bad set. By compactness the sequence $(v_{n})$ has a convergent subsequence whose limit belongs to $A_{1}\times \dots\times A_{m}\setminus U$, a contradiction. 

Therefore, for sufficiently large $n$, we have $W_{1,n}\times\dots\times W_{m,n}\subseteq U$. If the $A_{i}$ are all equal then for each $n$, the $W_{i,n}$ are all equal.
\end{proof}

\begin{lemma}\label{sec2:lemma3}
Let $B$ be the set of ordered $m$-tuples $(x_{1},\dots,x_{m})$ of $\underbrace{S^{n-1}\times\dots\times S^{n-1}}_{m}$ such that the set 
$\{x_{1},\dots,x_{m}\}$ is balanced. Then the following hold.
\begin{enumerate}
\item
$B$ is closed.
\item
A subset $A$ of $S^{n-1}$ is $m$-tame if and only if $\underbrace{A\times\dots\times A}_{m}$ is disjoint from $B$.
\item
If $A$ is a closed $m$-tame subset of $S^{n-1}$ then $A$ has an open neighbourhood whose closure is also $m$-tame.
\end{enumerate}
\end{lemma}
\begin{proof}
(i) and (ii) are clear. To establish (iii), use Lemma \ref{sec2:lemma2} together with (i) to separate $A\times\dots\times A$ and $B$ by open sets, with the open neighbourhood of $A\times\dots\times A$ being of the form $U\times\dots\times U$ for some $U\supseteq A$. Then the closure $\overline{U\times\dots\times U}=\overline U\times\dots\times\overline U$ is again disjoint from $B$ and $U$ has the desired properties by (ii).
\end{proof}

We note that if the set $U$ in Theorem B1 has a finite number $m$ of components then it follows from the proof that it is contained in a weakly maxtame set with at most $m$ components.

As promised in the introduction, we include a proof of the basic and doubtless well known  
\begin{lemma}\label{completelytamecase}
If $U$ is either open or closed in $S^{n-1}$ and is $(n+1)$-tame then
$U$ is contained in an open hemisphere.
\end{lemma} 
\begin{proof}
A classical result of Carath\'eodory states that the convex hull of a subspace of Euclidean space $\R^{n}$ is the union of the convex hulls of the $(n+1)$-point subsets, see (\cite{Caratheodory}, Satz 9, end of the third paragraph).

Let $U$ be $(n+1)$-tame in $S^{n-1}$. Let $V$ be the convex hull of $U$ in the ambient Euclidean space $\R^{n}$. The tameness condition ensures that none of the convex hulls of $(n+1)$-point subsets of $U$ contain the origin. Therefore, by Carath\'eodory's theorem, $V$ is a convex set in $\R^{n+1}$ not containing the origin. The closure $\overline V$ of $V$ either does not contain the origin or contains the origin on its boundary. It follows that $\overline V$ is contained in a closed half-space: this can be deduced from (\cite{Rockafellar1970},Theorem 11.5) for example.  If $U$ is open then $V$ is open and is contained in the interior open half-space.
If $U$ is closed then the origin does not belong to $V$ and one can again deduce that $V$ is contained in an open half-space.
\end{proof}

\section{Polyhedrality of weakly maxtame sets}

The next result is our tool for establishing polyhedrality.
For $v\in S^{n-1}$, we use the notation $H(v)$, $\overline H(v)$ to denote the open  and closed hemispheres in $S^{n-1}$ centred at $v$.

\begin{lemma}[The Key Lemma]\label{keylemma}
Let $X\subseteq U$ be subsets of $S^{n-1}$ such that $U$ is either open or closed and $n$-tame, and $X$ is closed and $(n+1)$-tame. Then there is an open hemisphere $H$ containing $X$ such that $U\cap \partial H=\emptyset$.
\end{lemma}
\begin{proof}
Suppose first that $U$ is closed. Consider the set $\mathcal X$ of subsets of $U$ which contain $X$ and which are contained in a closed hemisphere. We use Zorn's Lemma to show that $\mathcal X$ contains maximal elements. 
Lemma \ref{completelytamecase} shows that $X$ is contained in a closed hemisphere and so $X\in\mathcal X$ and $\mathcal X$ is non-empty.
Suppose now that $\mathcal Y$ is a totally ordered subset of $\mathcal X$. For any $Y\subseteq S^{n-1}$ let $N_{Y}$ denote the set of $v\in S^{n-1}$ such that $Y\subseteq\overline H(v)$. Then $N_{Y}$ is closed and for $Y\subseteq Y'$ we have $N_{Y}\supseteq N_{Y'}$. As $Y$ runs through $\mathcal Y$, the $N_{Y}$ are a family of non-empty closed sets with the finite intersection property. Therefore there exists $u\in\bigcap_{Y\in\mathcal Y}N_{Y}$. Now $\bigcup_{Y\in\mathcal Y}Y\subseteq\overline H(u)$ and hence totally ordered sets in $\mathcal X$ are bounded above.

Let $Z$ be a maximal element of $\mathcal X$ and let $H$ be a closed hemisphere which contains $Z$. We have $Z\subseteq H\cap U\in\mathcal X$ and so the maximality of $Z$ guarantees that $Z=H\cap U$ and $Z$ is closed. Consider the intersection $W:=Z\cap\partial H$. Then $W$ is an $n$-tame subset of the $(n-2)$-sphere $\partial H$ and hence it is contained in an open hemisphere $K\subset\partial H$. Now we can define a small rotation $R$ fixing $\partial K$ so that $Z$ lies in the interior of the hemisphere $R(H)$. Maximality of $Z$ now ensures that $U\cap\partial R(H)=\emptyset$.

Suppose now that $U$ is open.
Using Urysohn's Lemma, we may choose a continuous function $f:S^{n-1}\to[0,1]$ such that 
\begin{itemize}
\item
$f$ vanishes on $S^{n-1}\setminus U$;
\item
$f$ takes value $1$ on $X$; and
\item
$f$ takes values in the open interval $(0,1)$ on $U\setminus X$.
\end{itemize}
For each natural number $\ell$ let $U_{\ell}$ denote the open set $f^{-1}\left((\frac1{\ell+1},1]\right)$ and let $V_{\ell}$ denote the closed set $f^{-1}\left([\frac1{\ell+1},1]\right)$. Then we have 
\[ X\subset U_{1}\subset V_{1}\subset U_{2}\subset V_{2}\subset\dots \]
and $U=\bigcup_{\ell=1}^{\infty}V_{\ell}$. By the closed case already proved we know that 
\[ D_{\ell}=\{v\in S^{n-1};\ X\subseteq H(v)\textrm{ and } V_{\ell}\cap \partial H(v)=\emptyset\} \]
is non-empty for each $l$. Therefore
\[ E_{\ell}=\{v\in S^{n-1};\ X\subseteq H(v)\textrm{ and } U_{\ell}\cap \partial H(v)=\emptyset\} \]
is non-empty for each $l$. The sets $E_{\ell}$ are closed and nested ($E_{\ell}\supseteq E_{\ell+1}$) so by compactness there is an $x$ in their intersection. For this $x$ we have $X\subseteq H(x)$ and $U_{\ell}\cap \partial H(x)=\emptyset$ for all $\ell$. Since $\bigcup U_{\ell}=U$ the result follows.
\end{proof}

\begin{definition}
Let $U,V$ be closed polyhedral subsets of $S^{n-1}$. We shall say that $U$ \emph{is adjacent to} $V$ if and only if $U\cap V$ has a codimension one component.
\end{definition}

\begin{definition}
Let $U$ be an $n$-tame subset of $S^{n-1}$. We shall say that a point $x\in S^{n-1}$ is \emph{addable (with respect to $U$)} if and only if $U\cup\{x\}$ is $n$-tame. We write $U^{+}$ for the set of all addable points.
\end{definition}

We shall need to work with convex sets in spherical geometry. A little care needs to be taken when assigning meaning to the spherical convex hull of a set whose Euclidean convex hull contains the origin (i.e. the centre of the sphere). We shall adopt the convention that the spherical convex hull of $U\subseteq S^{n-1}$ is the projection from the origin to the sphere of the Euclidean convex hull in case the Euclidean convex hull does not contain the origin and is $S^{n-1}$ itself otherwise.

\begin{notation}
Let $U$ be a subset of $S^{n-1}$. We write $U[m]$ for the union of the spherical convex hulls of the $m$-point subsets of $S^{n-1}$. This is called the \emph{$m$-hull of $U$}.
\end{notation}

\begin{lemma}\label{whatisaddable}
Let $U$ be an $n$-tame subset of $S^{n-1}$. Then $U^{+}=S^{n-1}\setminus -U[n-1]$.
\end{lemma}
\begin{proof}
If a point $x\in S^{n-1}$ belongs to a balanced $n$-point subset $E\subseteq U\cup\{x\}$ then $-x$ belongs to the spherical hull of $E\setminus\{x\}$ and this is contained in $\subseteq U[n-1]$. Thus $x\in -U[n-1]$. The converse is equally straightforward.
\end{proof}

\begin{lemma}\label{nudge}
Let $U$ be an open $n$-tame subset of $S^{n-1}$ and let $H$ be an open hemisphere. Then $U\cup(H\cap\overline U)$ is $n$-tame.
\end{lemma}
\begin{proof}
Suppose not. Let $F$ be a balanced $n$-point subset of 
$U\cup(H\cap\overline U)$. Set $E:=F\setminus U$. Since $E$ is contained in the open hemisphere $H$ it follows that at least one point $u$ of $F$ belongs to $U$. Each of the points of $E$ belongs to the closure of $U$. By nudging the points of $E$ back into $U$ and allowing the point $u$ to compensate by adjusting its position in the open set $U$ we then obtain a balanced $n$-point subset of $U$, a contradiction.
\end{proof}

It follows in particular that every point in the closure of $U$ is addable.

\begin{lemma}\label{peter}
Let $U$ be an $n$-tame subset of $S^{n-1}$ and let $C$ be a non-empty path-connected subset of $U$. Then $U\cup C[2]$ is $n$-tame.
\end{lemma}
\begin{proof}
For each finite set $F$, write $\mu F$ for the number of points not in $U$. For a contradiction, suppose that there is a balanced $n$-point subset of $U\cup C[2]$. Let $F$ be a choice of such a set with $\mu F$ as small as possible. Clearly $F$ contains at least one point $w$ not in $U$. Then $w$ belongs to $C[2]$ and we may choose two points $u,v$ in $C$ such that $w$ lies on the short geodesic joining $u$ to $v$.  We note by tameness of $U$, $u$ and $v$ are not antipodal.
Let $E:=(F\setminus\{w\})\cup\{u,v\}$. We have $\mu E=\mu F-1$ and therefore none of the $n$-point subsets of $E$ are balanced. It follows that $E$ spans a simplex in Euclidean space which contains the origin in its interior. Let $p:[0,1]\to C$ be a path from $u$ to $v$. Consider the simplices $\Sigma(t)$ spanned by $E(t):=(F\setminus\{w\})\cup\{u,p(t)\}$. We have $E(1)=E$ while $E(0)$ is a face of $E$.  It follows that for some $t$, $E(t)$ has a face which contains the origin. At this time, the vertices of this face form a balanced set. However, $\mu E(t)$ remains always strictly less that $\mu F$ because $p(t)$ is contained in $u$. This is a contradiction.
\end{proof}

Applying the lemma iteratively we can replace $C[2]$ by its spherical hull, although we will not directly use this fact.

\begin{proof}[Proof of Theorem A]
Let $U$ be a weakly maxtame subset of $S^{n-1}$. 
Let $F$ be a finite set of representatives of the path components of $U$.
We begin by finding a finite set $\mathcal K$ of open hemispheres such that
\begin{itemize}
\item
every $n$-point subset of $U$ is contained in some $K\in\mathcal K$; and
\item $U\cap\partial K=\emptyset$ for all $K\in\mathcal K$.
\end{itemize}
This is achieved by choosing one hemisphere $K=K(E)$ for each $n$-point subset $E$ of $F$ using Key Lemma \ref{keylemma}. Then $\mathcal K$ has at most $\binom{|F|}{n}$ hemispheres. Given an arbitrary $n$-point subset $G$ of $U$, let $E\subseteq F$ be the corresponding representatives. Connectivity ensures that $K(E)$ contains $G$ and the first condition is already satisfied.

Now consider the tiling of $S^{n-1}$ produced by the great subspheres $\partial K$ as $K$ runs through $\mathcal K$. Let $X$ denote the union of these subspheres. Let $V$ denote the union of those components of the complement $X^{c}$ which meet $U$. The same set $F$ represents the components of $V$ and the same set of hemispheres is witness to the fact that $V$ is $n$-tame. Therefore, by maximality, $V=U$. It is now clear that each component is polyhedral.

Let $\mathcal H$ be the set of those members $H$ of $\mathcal K$ with the property that some component of $\overline U$ is adjacent to $\partial H$. Part (i) of the theorem holds for $\mathcal H$.

Suppose now that $H\in\mathcal H$ as above.  Let $B_r(p)$ denote the ball in $S^{n-1}$ of radius $r$ about the point $p$.  By (i) and the adjacency condition there is some point $p\in \partial H$ and $r>0$ with $B=B_r(p)\cap H$ contained in $U$.  For each $n\in\N$ the open set $U\cup B_{r/n}(p)$ is not $n$-tame by maximality. Thus there is an $n$-tuple $F_n\in S^{n-1}\times\dots\times S^{n-1}$, whose entries form a balanced set in $U\cup B_{r/n}(p)$, at least one of whose points is not in $U$.  We may choose a convergent subsequence whose limit $F$ gives a balanced set of $m\le n$ elements, with $p\in F$.  By construction the points of $F\setminus\{p\}$ are in the closure of $U$.  We have established that $U$ is a finite union of polyhedral components so we may choose arcs $\gamma_i:[0,1)\rightarrow U$, where $i$ runs from $1$ to $m$, with $F\setminus\{p\}=\{\gamma_1(1),\dots,\gamma_m(1)\}$.

Note that each $\gamma_i$ is contained in $\overline H$; otherwise as in the proof of Lemma \ref{nudge} we could nudge $p$ into $B$ and each $\gamma_i(1)$ into $U$ giving a balanced $m$-point set in $U$.  Let $W$ be the union of $B$ and the arcs $\gamma_i$; it follows that $W[n-1]$ contains a neighbourhood $A$ of $\partial H$ in $\overline H$.  The union $N$ of $\overline H$ and $-A$ satisfies condition (ii) of the theorem.

For the last sentence of the theorem suppose $C_1$, $C_2$ are two components of a weakly maxtame set $U$ with $c$ in the intersection of $\overline{C_1}$ and $\overline{C_2}$.  
Let $C=C_1\cup C_2$.  By the remark after Lemma \ref{nudge}, $c$ is addable so that $U\cup\{c\}$ is $n$-tame.  Note that $C\cup\{c\}$ is path-connected. By Lemma \ref{peter}, the union of $U$ with $(C\cup\{c\})[2]$ is $n$-tame and therefore so is the open set
$U\cup C[2]$; this violates maximality of $U$.
\end{proof}

\section{Maxtame sets on the $2$-sphere: constructions and examples}

We begin with the second part of Theorem B, an effective method for constructing maxtame sets on a $2$-sphere from any weakly maxtame starting point. For clarity, here is a self-contained statement of the theorem. It is the case $n=3$ of Conjecture B.

\begin{theoremb2}
Let $U$ be a $3$-tame subset of $S^{2}$ which is either
\begin{enumerate}
\item
closed; or
\item
open with finitely many connected components.
\end{enumerate}
Then $U$ is contained in a maxtame subset.
\end{theoremb2}

In the proof we shall use the case $n=3$ of the following result.

\begin{lemma}
Let $X$ be a subset of $S^{n-1}$ which is $n$-tame but not $(n+1)$-tame. Then every path component of $S^{n-1}\setminus X[n-1]$ is convex.
\end{lemma}
\begin{proof}
This argument rests on the fact that every point of the sphere has a natural \emph{rank with respect to $X$}. For $0\le\ell\le n$ we say that $v\in S^{n-1}$ has rank $\ell$ if and only if $v$ belongs to $X[\ell+1]$ but does not belong to $X[\ell]$. The hypothesis that $X$ is not $(n+1)$-tame ensures that $X[n]=S^{n-1}$ and hence every point of the sphere has a rank $\ell$ in the range $0\le\ell\le n-1$. The points of $X$ itself are the points of rank $0$. The set $S^{n-1}\setminus X[n-1]$ consists precisely of the points of rank $n-1$. 

To prove the lemma we first fix a balanced $(n+1)$-point subset $F$ of $X$, i.e. a witness to the hypothesis that $X$ is not $(n+1)$-tame. Since $X$ is $n$-tame, each $n$-point subset of $F$ has a spherical hull which is a face of a simplex triangulating $S^{n-1}$.

For any $n$-point subset $E$ of $X$, the spherical convex hull $\Delta$ of $E$ is a simplex properly contained in the sphere. Taking the points of $F$ one by one, we may build up a simplicial triangulation of $S^{n-1}$ in which $\Delta$ is a face. At each step, ignore points which lie in the spherical hull of the complex so far constructed. To add a point $x$ of $F$ when it does not belong to that hull we simply cone it off to those faces of the simplicial complex already constructed that are \emph{visible} from $x$, in the sense that there are short geodesics from these faces to $x$ which contain no points of the complex in their interior.  At last there must come a point of $F$ which, when added, results in a complex that covers the entire sphere since, with $E=\emptyset$, adding all points of $F$ would achieve this.

Now suppose that $v$ and $w$ are points of rank $n-1$ 
in the same path component of 
$S^{n-1}\setminus X[n-1]$.  If $u$ is any point of $X[n-1]$ then the above argument shows that $u$ lies in the $(n-2)$-skeleton of some triangulation of the sphere whose vertices belong to $X$. The $(n-1)$-faces  of this triangulation are convex and are the path components of the complement of its $(n-2)$-skeleton. Therefore $v$ and $w$ both belong to the same convex component and hence the short geodesic between them lies within that component and so does not contain $u$. This shows that no point of $X[n-1]$ lies on the geodesic from $v$ to $w$ and the convexity follows.
\end{proof}

The following is immediate on combining this with Lemma \ref{whatisaddable}.

\begin{corollary}\label{convexity}
Let $U$ be an $n$-tame subset of $S^{n-1}$ which is not $(n+1)$-tame. Then each connected component of the set of addable points is convex.
\end{corollary}

\begin{remark}
If $U$ is open and $(n+1)$-tame then $U$ is contained in an open hemisphere. In this case, the set of addable points is the union of all closed hemispheres containing $U$ and it is equal to the complement of the spherical convex hull of $-U$. In particular the set of addable points of a non-empty $(n+1)$-tame set consists of a single component and is never convex.
\end{remark}

\begin{proof}[Proof of Theorem B2]
The reader will notice that much of the following reasoning could be applied in higher dimensions. However, at a critical point of the proof we have found the need to take advantage of the low dimension: we draw attention to this point when it arises. 

By Lemma \ref{whatisaddable} if $U$ is $3$-tame in $S^{2}$ then the set $U^{+}$ of addable points is equal to $-U[2]^{c}$, the complement of the $2$-hull of the antipodal set of $U$. If $U$ is a union of convex spherical polygons, (for example, if $U$ is weakly maxtame) then clearly $U^{+}$ is as well. 

Using Theorem B1, we may assume that $U$ is weakly maxtame and in particular that $U$ is a disjoint union of convex spherical polygons.  
If $U$ is $4$-tame then $U$ is contained in an open hemisphere and we are done. Therefore we may assume that there is a balanced $4$-point subset $F=\{a,b,c,d\}$. Then $F$ must consist of exactly $4$ points and its convex hull in Euclidean space is a tetrahedron which contains the origin in its interior. Tameness guarantees that $U$ is contained in the complement of $-F[2]$, the $2$-hull of the antipodal set of $F$. We naturally regard $-F[2]$ as the $1$-skeleton of a simplicial subdivision of $S^{2}$. The (open) faces $\Phi_{1},\Phi_{2},\Phi_{3},\Phi_{4}$ contain the set $U^{+}$ of all addable points. We now show how to replace $U$ by a larger open set $U_{0}$ which is still polyhedral and $n$-tame so that 
$U\subseteq U_{0}\subset U_{0}^{+}\subseteq U^{+}$ and so that $\Phi_{1}$ does not meet the interior of $U_{0}^{+}\setminus U_{0}$. Let $C_{1},\dots,C_{\ell}$ be an enumeration of the connected components of the interior of $U^{+}\setminus U$ that lie in $\Phi_{1}$. Note that $U\cup C_{1}$ is $n$-tame, for if not then there would be a balanced set containing at least one point of $C_{1}$. However, such a balanced set can always be replaced by one that contains a single point of $C_{1}$ because $C_{1}$ is convex and since $C_{1}$ consists entirely of addable points, this single point can be added without creating a balanced set. Therefore we replace $U$ by $U_{(1)}:=U\cup C_{1}$. The set of addable points now becomes potentially smaller but remains polyhedral. Therefore when we turn to the component $C_{2}$ it maybe that the interior of $C_{2}':=C_{2}\cap U_{(1)}^{+}$ is properly contained in $C_{2}$ and furthermore, $C_{2}'$ might consist of more than one component. 

At this point in the argument we need to use the hypothesis that $n=3$. Therefore, for the remainder of the proof we shall make this assumption.

The set $U_{(2)}:=U_{(1)}\cup(U_{1}^{+}\cap C_{2})$ is $3$-tame. To see this suppose for a contradiction that there is a balanced $3$-point subset inside $U_{(2)}$. Then at least one of the $3$ points must lie outside $U_{1}$ because $U_{(1)}$ is $3$-tame. Moreover, if only one of the three points lies outside $U_{(2)}$ then there can be no contradiction beause each one point is addable. If all three points lie outside of $U_{(2)}$ then we would have all three points lying on the face $\Phi_{1}$ and each face is contained in an open hemisphere. So there must be exactly two points $c,c'$ outside $U_{(2)}$ and one point $u$ inside $U_{(2)}$. The balanced condition means antipodal point $-u$ lies on the geodesic joining $c$ to $c'$ and therefore $-u$ lies in the convex set $C_{2}$. Since $\Phi_{1}$ is contained in an open hemisphere and $-u$ lies in this face, it follows that $u$ does not lie in $\Phi_{1}$. In particular, $u$ is not in the set $C_{1}$ and therefore $u$ belongs to the original set $U$. However, every point of $C_{2}$ was addable with respect to $U$ and so $-u$ is addable, a contradiction. Thus $U_{(2)}$ is $3$-tame. This part of the reasoning does not generalize straightforwardly to higher dimensions.

We now turn to $C_{3}$. The process of adding $C_{1}$ and then adding the splinters of $C_{2}$ that remained addable after adding $C_{1}$ may have splintered $C_{3}$ into many components, but the same argument allows us to add the interiors of these all in one step.

Continuing in this way, we may visit each of the sets $C_{1},\dots,C_{\ell}$ in turn, adding what can be added, preserving polyhedrality and arriving at a situation where the interior of the set of addable points has empty intersection with $\Phi_{1}$. This new set may be labelled $U_{0}$. 

We now turn to the next face $\Phi_{2}$ of $\Delta$. The process carried through for $\Phi_{1}$ may have resulted in the number of components of addable points on $\Delta$ increasing dramatically. This has not mattered because it is only now that we enumerate the components and proceed as with the first face. Going on in this way, we can complete the process by visiting each face of $\Delta$ and arriving at a maxtame set as required.

\end{proof}

We conclude this section with some examples of maxtame sets on $S^{2}$. 
The simplest example is shown in Figure \ref{fig1}. The construction of this example is simple and can be generalized in a number of ways. Any choice of $4$ great circles in general position divides the sphere into $14$ regions: $8$ spherical triangles and $6$ spherical quadrilaterals. There are two ways (antipodally related to each other) to choose $4$ of the triangles in order to obtain a maxtame set. Figure \ref{fig1} shows one of these choices. The example is the first in two infinite families of maxtame sets on $S^{2}$: see Figure \ref{fig2} (Family A) and Figure \ref{fig3} (Family B).

\begin{figure}[htbp]
   \centering
   \includegraphics[width=2in]{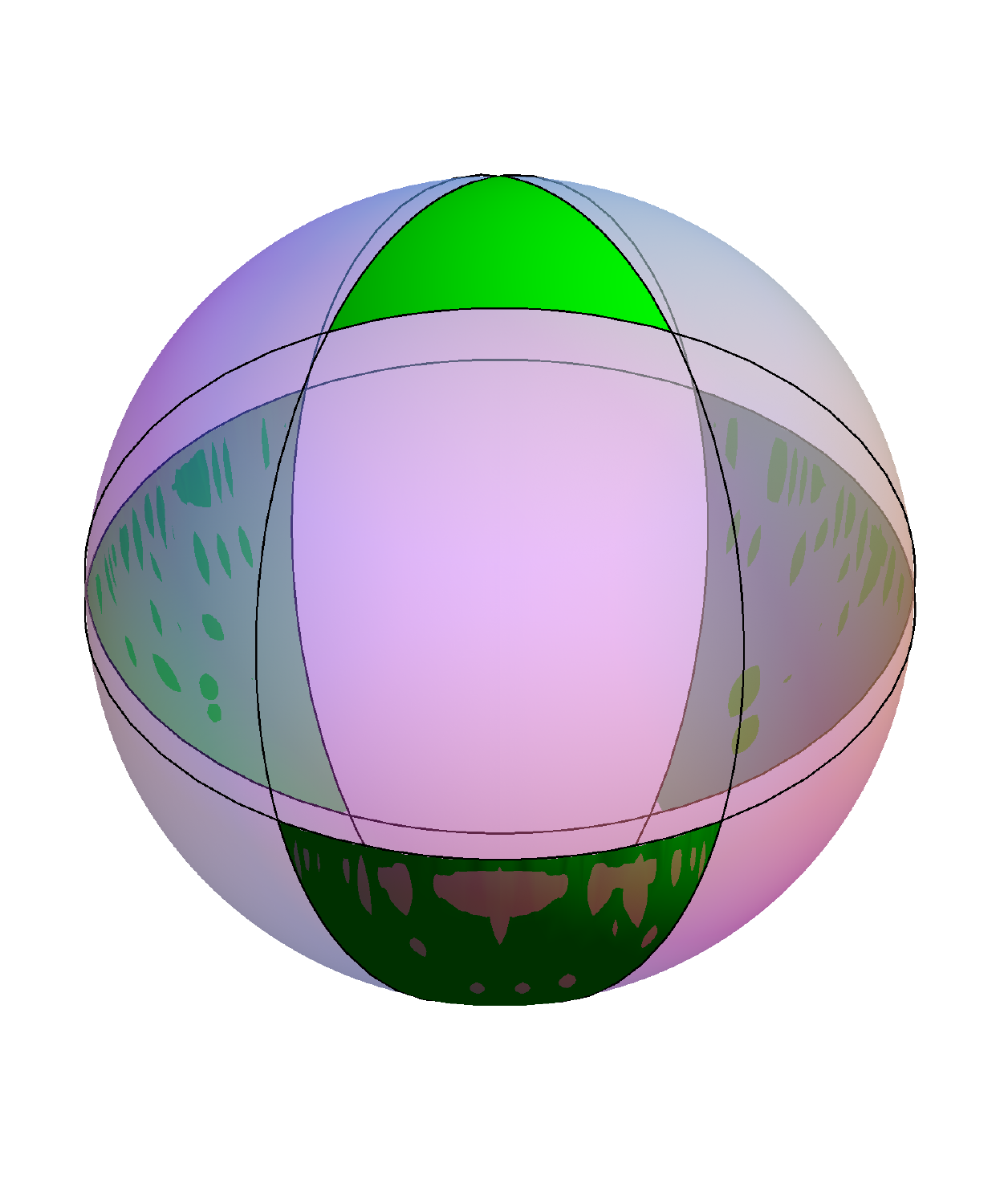} 
   \caption{A $4$-component maxtame set on $S^{2}$}
   \label{fig1}
\end{figure}

\begin{figure}[htbp] 
\begin{minipage}[b]{0.3\linewidth}
\centering
\includegraphics[scale=0.3]{S2-3T-3D-4Ca.pdf}
\end{minipage}
\begin{minipage}[b]{0.3\linewidth}
\centering
\includegraphics[scale=0.3]{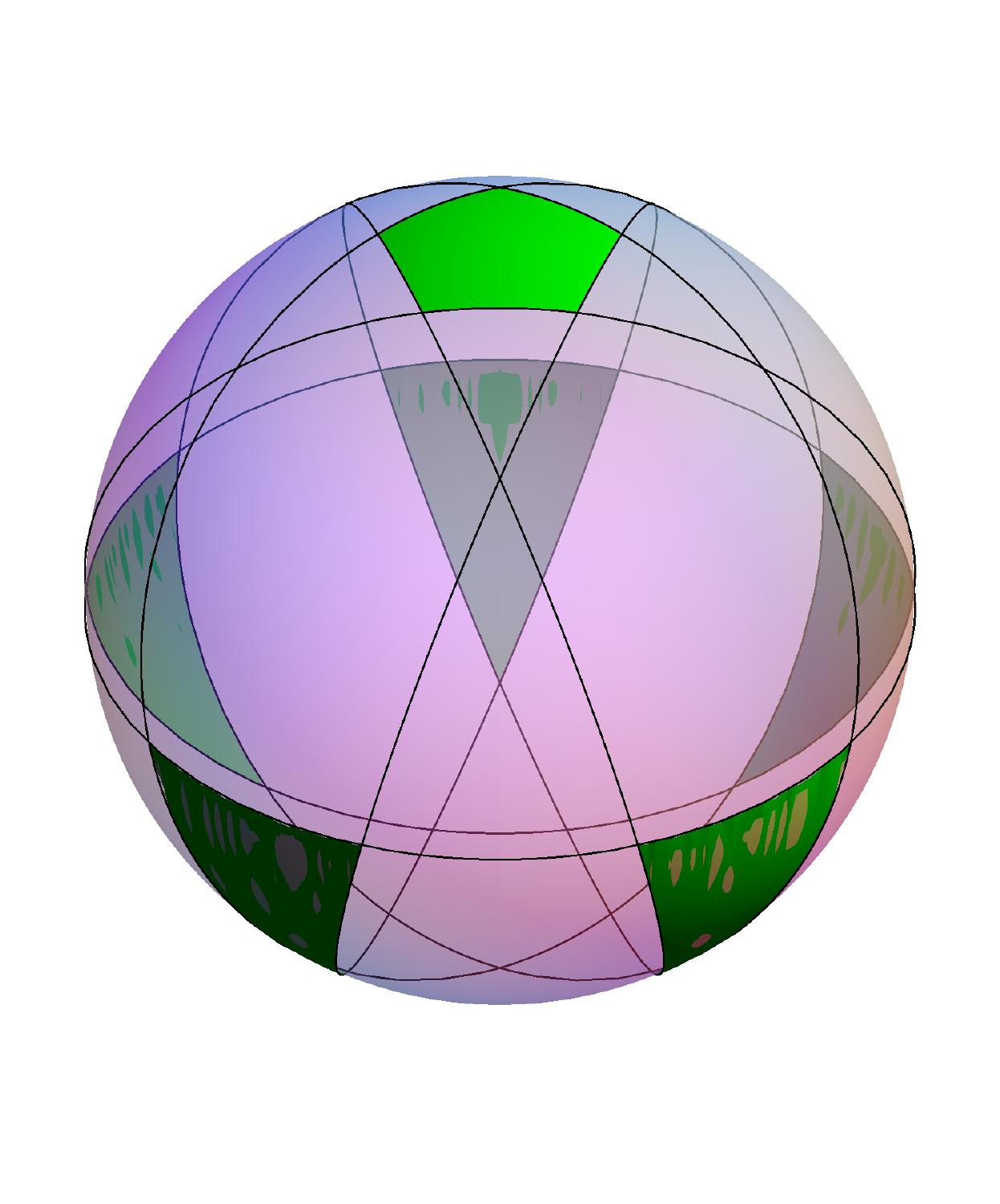}
\end{minipage}
\begin{minipage}[b]{0.3\linewidth}
\centering
\includegraphics[scale=0.3]{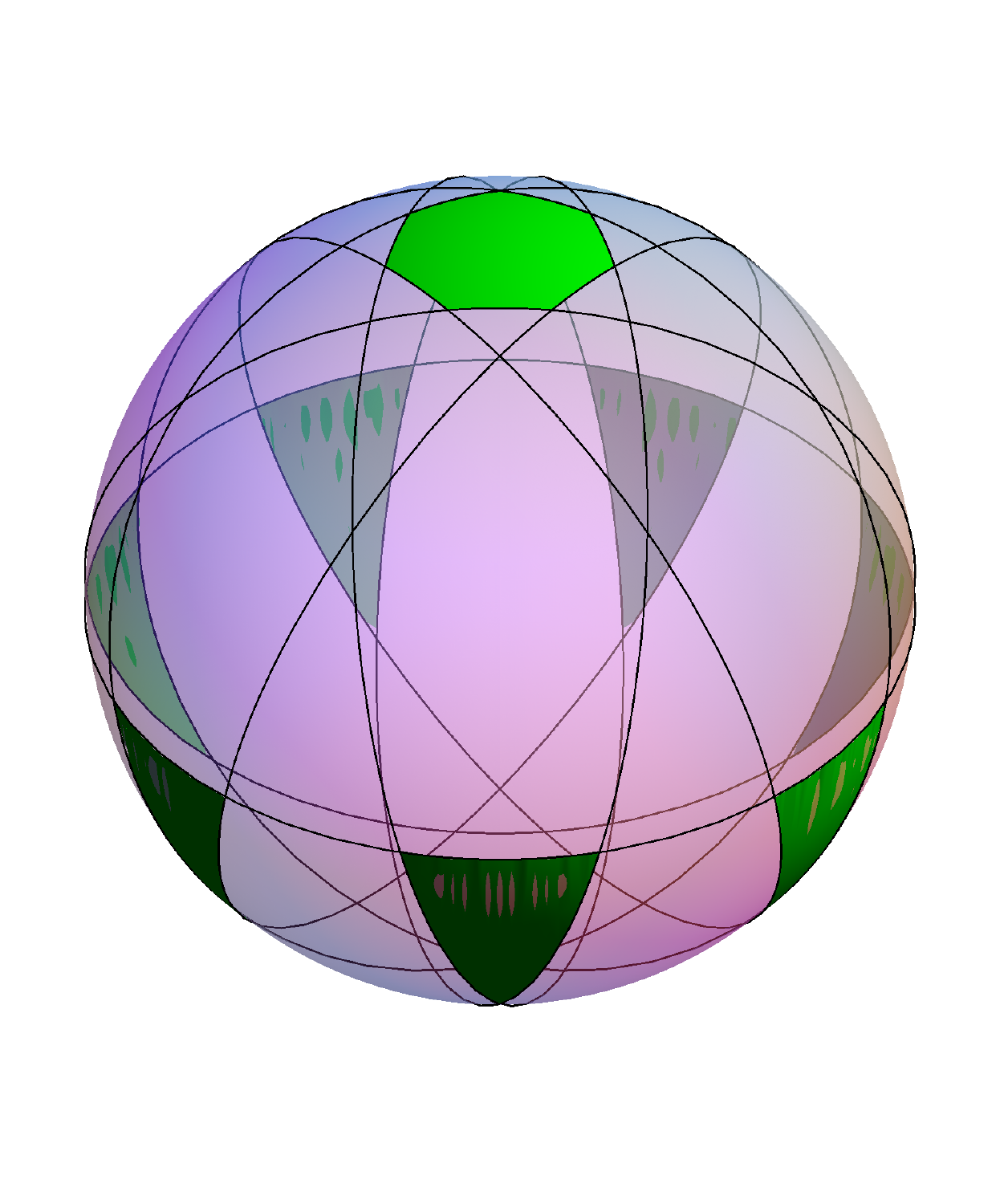}
\end{minipage}
\begin{minipage}[b]{0.3\linewidth}
\centering
\includegraphics[scale=0.3]{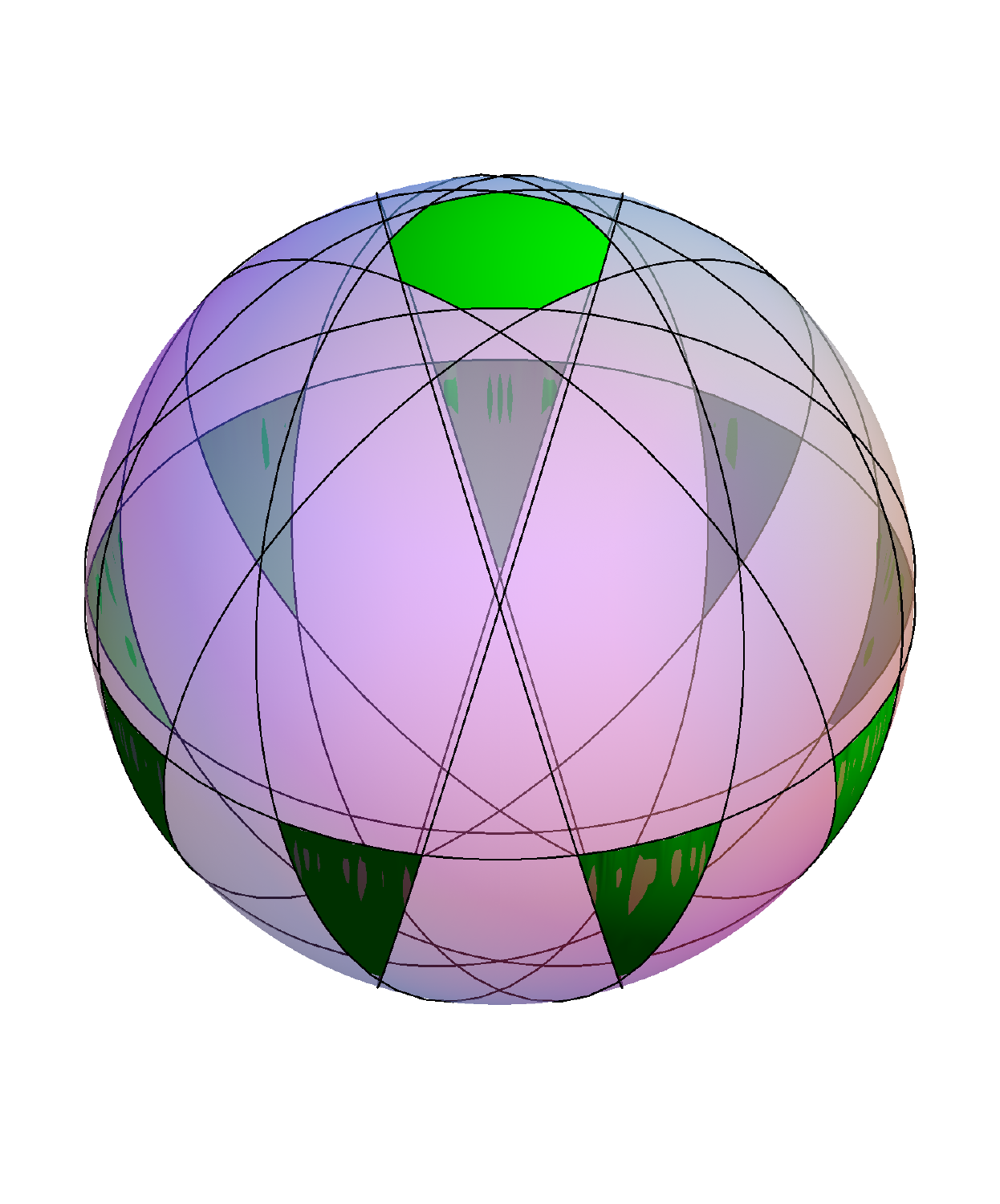}
\end{minipage}
\begin{minipage}[b]{0.3\linewidth}
\centering
\includegraphics[scale=0.3]{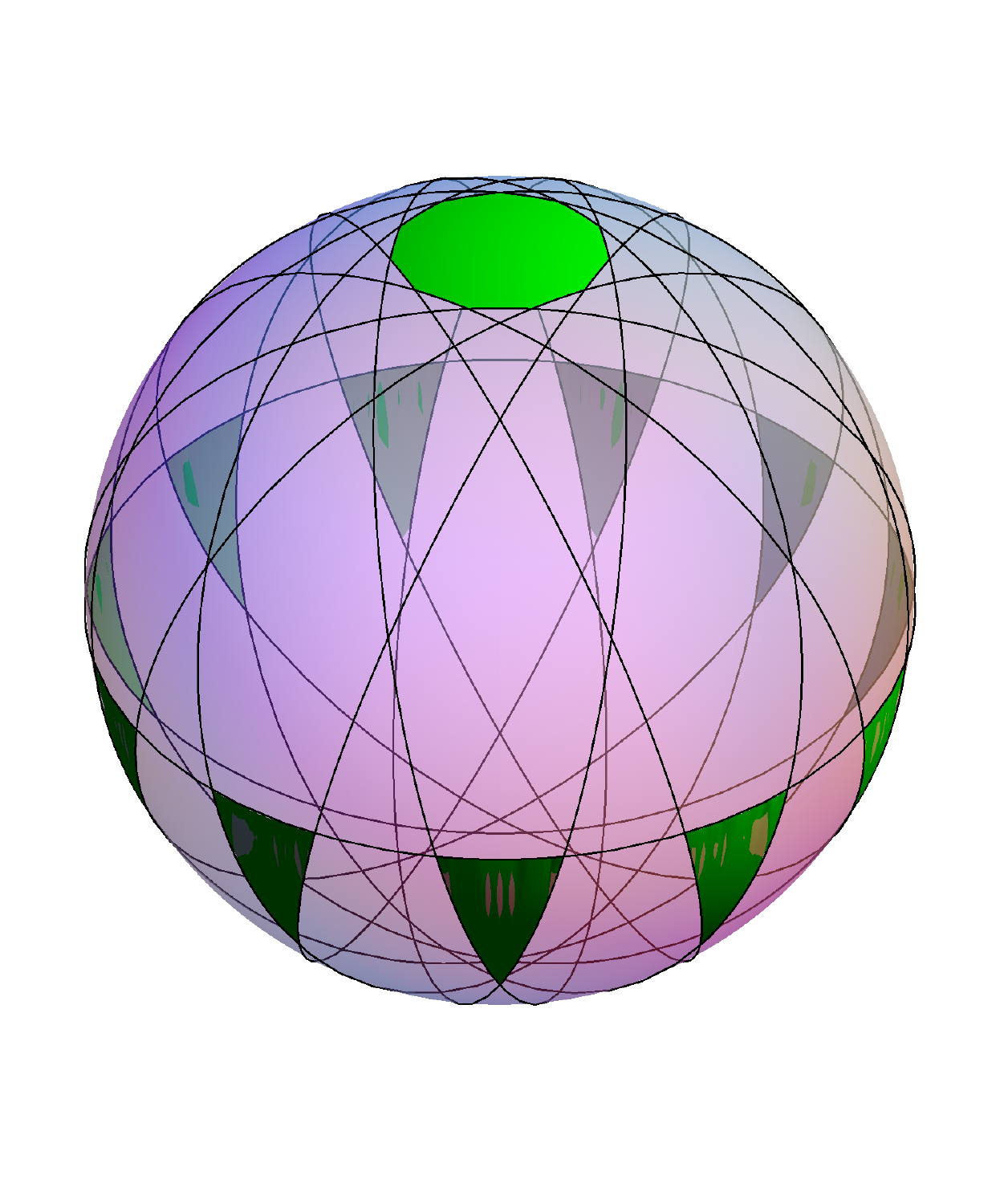}
\end{minipage}
\begin{minipage}[b]{0.3\linewidth}
\centering
\includegraphics[scale=0.3]{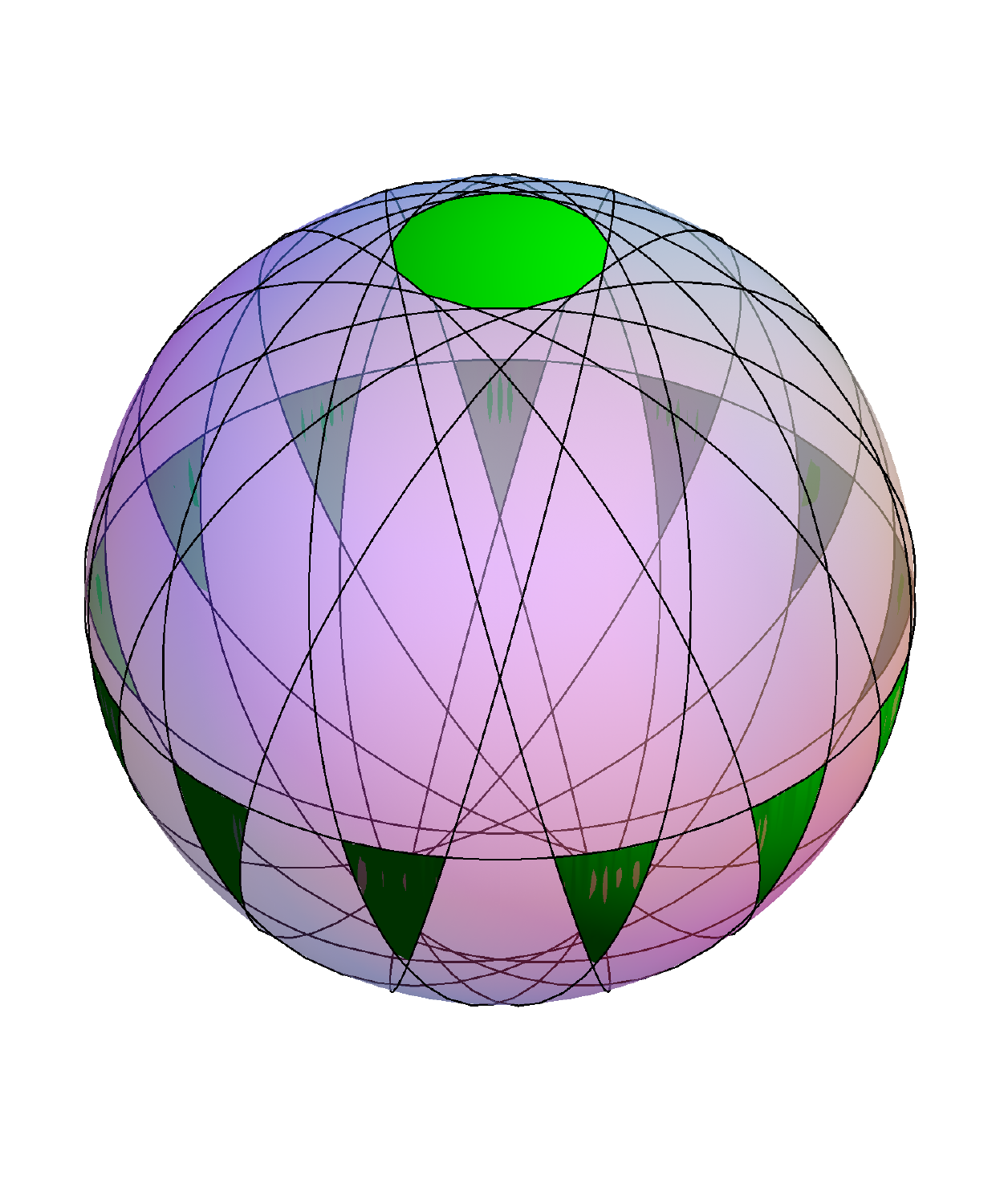}
\end{minipage}
   \caption{Family A: $3$-dimensional rendition. The $k$th member consists of a single $(2k+1)$-gon together with $2k+1$ triangles.}
   \label{fig2}
\end{figure}

\begin{figure}[htbp] 
\begin{minipage}[b]{0.4\linewidth}
\centering
\includegraphics[scale=0.35]{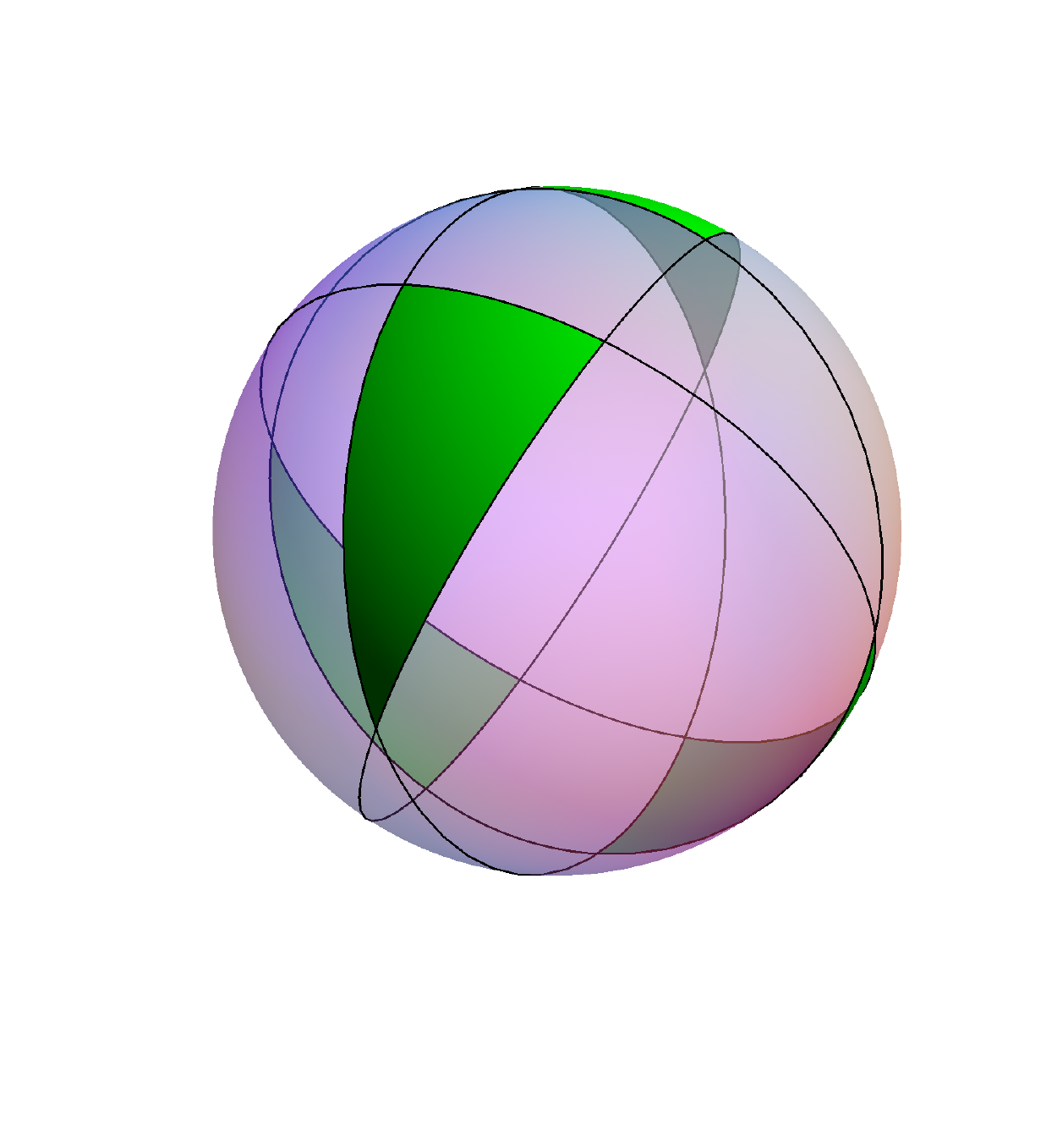}
\end{minipage}
\begin{minipage}[b]{0.4\linewidth}
\centering
\includegraphics[scale=0.35]{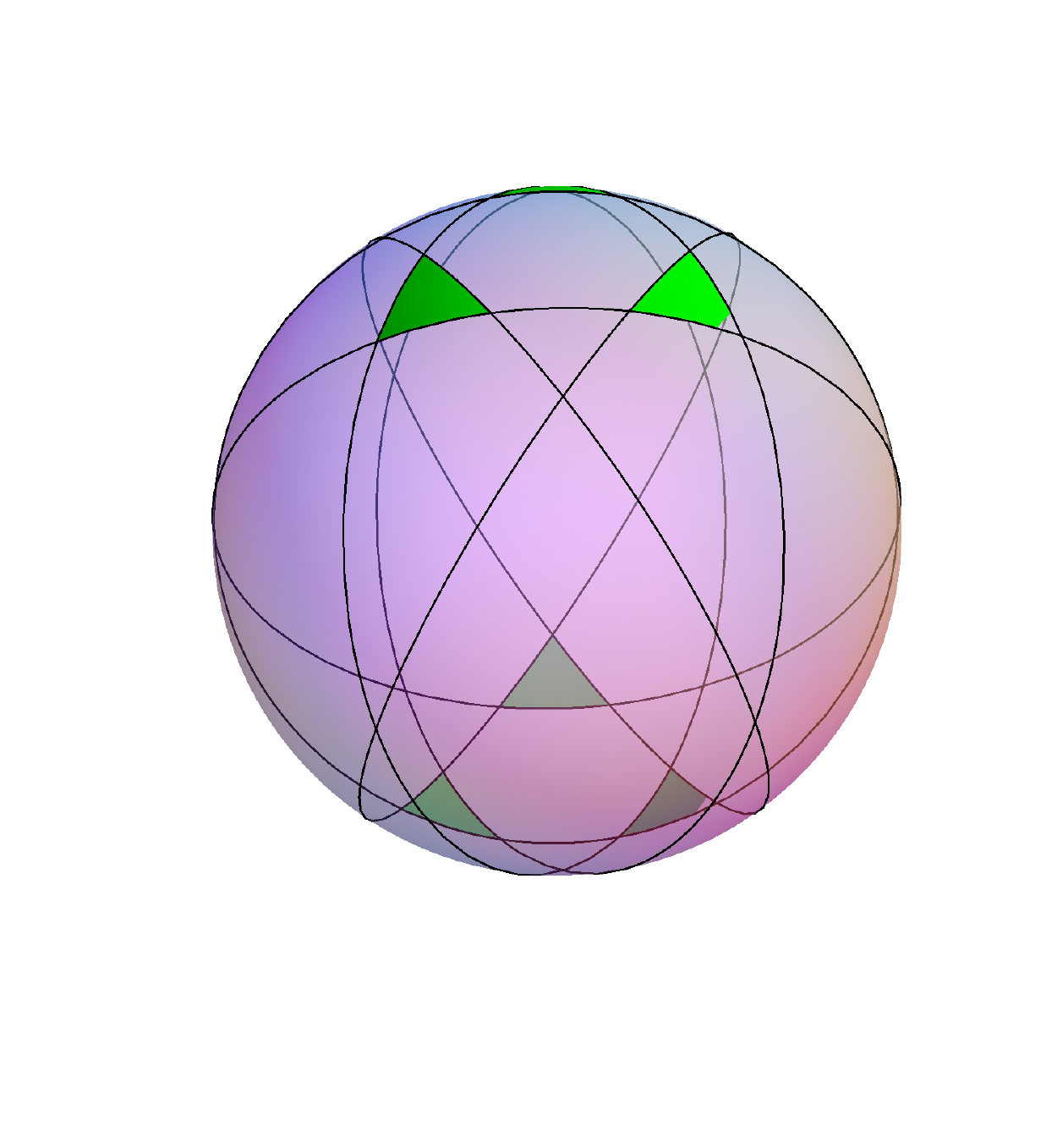}
\end{minipage}
\begin{minipage}[b]{0.4\linewidth}
\centering
\includegraphics[scale=0.35]{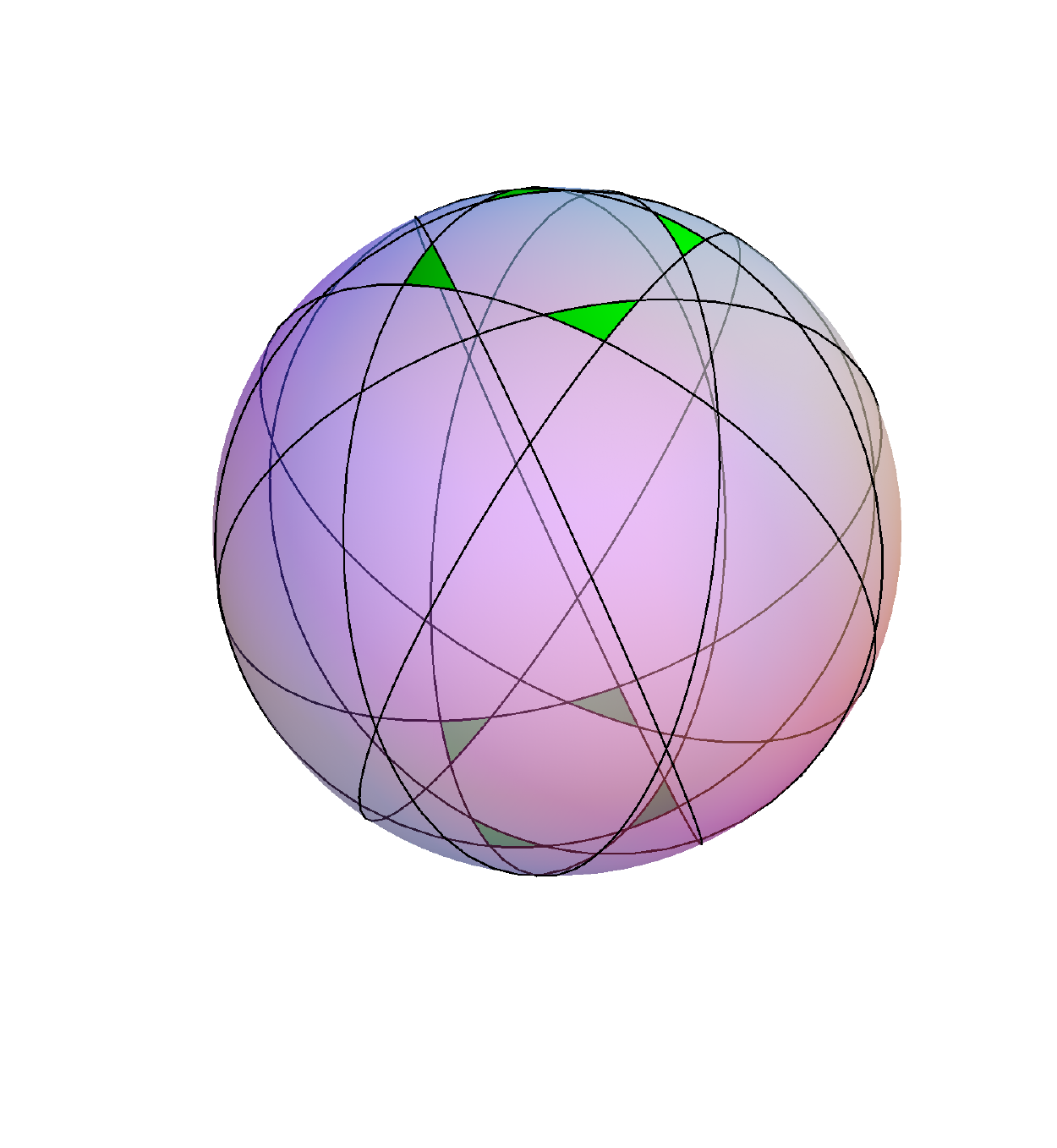}
\end{minipage}
\begin{minipage}[b]{0.4\linewidth}
\centering
\includegraphics[scale=0.35]{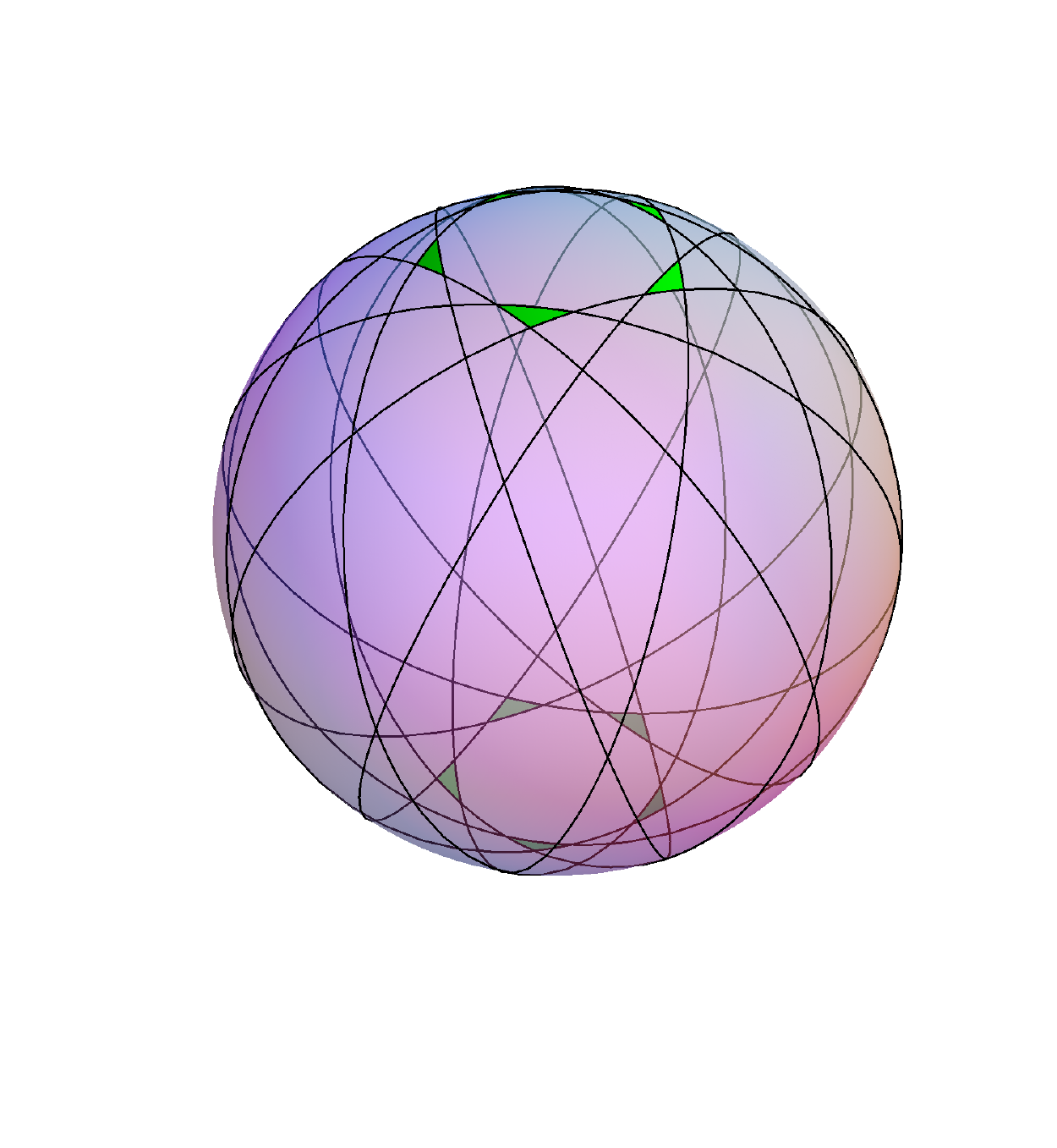}
\end{minipage}
   \caption{Family B: $3$-dimensional rendition. The $k$th member consists of $2k+2$ triangles.}
   \label{fig3}
\end{figure}

Figure \ref{fig4} shows these same sets in stereographic projection from the south pole.

In both cases it can be checked straightforwardly that the shaded set $U$ satisfies the condition $U=-\overline U[2]^{c}$ so that the interior of the set of addable points is exactly $U$. This allows us to deduce that all these sets are indeed maxtame as claimed. The sets in Family A satisfy the stronger condition, $\overline U=-U[2]^{c}$ meaning that the closure $\overline U$ is precisely the set of addable points. In the case of Family B, there are some isolated addable points: in particular, assuming that the set is drawn symmetrically around a north-south axis then the north and south poles are addable. Being isolated, these addable points do not allow the sets to be enlarged, i.e. they are already maxtame. However it is possible to perturb these sets in order to produce maxtame sets with additional regions but an otherwise similar appearance.

Checking that the sets in these Figures are $3$-tame and then maxtame can be done using the great circles in the picture. For example, $9$ 
great circles are shown in Figure \ref{fig5fig6}. If three points are chosen from the triangular regions then one can see by inspection that at least one of those great circles separates those points into a single hemisphere. Checking maxtameness involves constructing the antipodal set $-U$ and then its $2$-hull $-U[2]$. Again this is easy to carry out by inspection of the diagrams.

Figure \ref{fig5fig6} shows how a perturbation of the $6$-component member of Family B gives rise to a seventh addable region.
Pertubations of the other sets in Family B may give rise to other interesting examples with an odd number of components. We shall show later that there are restrictions: in particular there are no maxtame sets with $3$ or $5$ components.

\section{Maxtame sets on $S^{1}$ and their configuration space}

On the circle $S^{1}$ there are two maxtame sets associated with any choice of an odd number of distinct pairs of antipodal points as illustrated in Figure \ref{fig:2tameS1}.

\begin{figure}[htbp] 
   \centering
   \includegraphics[width=2.5in]{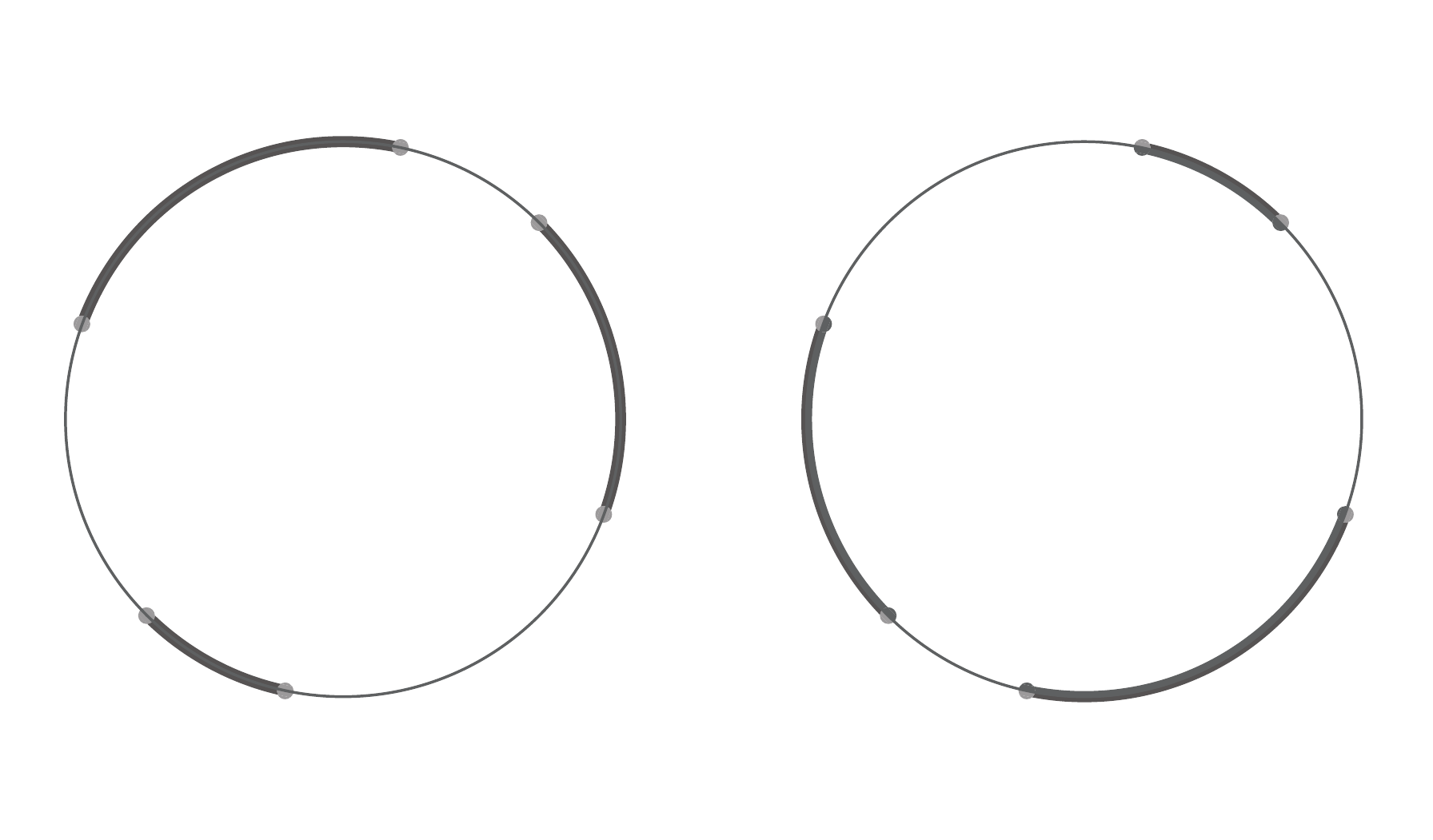} 
   \caption{Two maxtame sets on $S^{1}$ sharing the same $6$-point boundary}
   \label{fig:2tameS1}
\end{figure}

It is then interesting to consider the configuration space $\mathcal S$ of all such subsets. The double cover $S^{1}\to \RP^{1}$ induces a map from $\mathcal S$ to the configuration space of odd numbers of points on the circle $\RP^{1}$, and $\mathcal S$ is a connected double cover of this space.

\begin{landscape}
\begin{figure}[htbp] 
   \centering
   \includegraphics[width=7.1in]{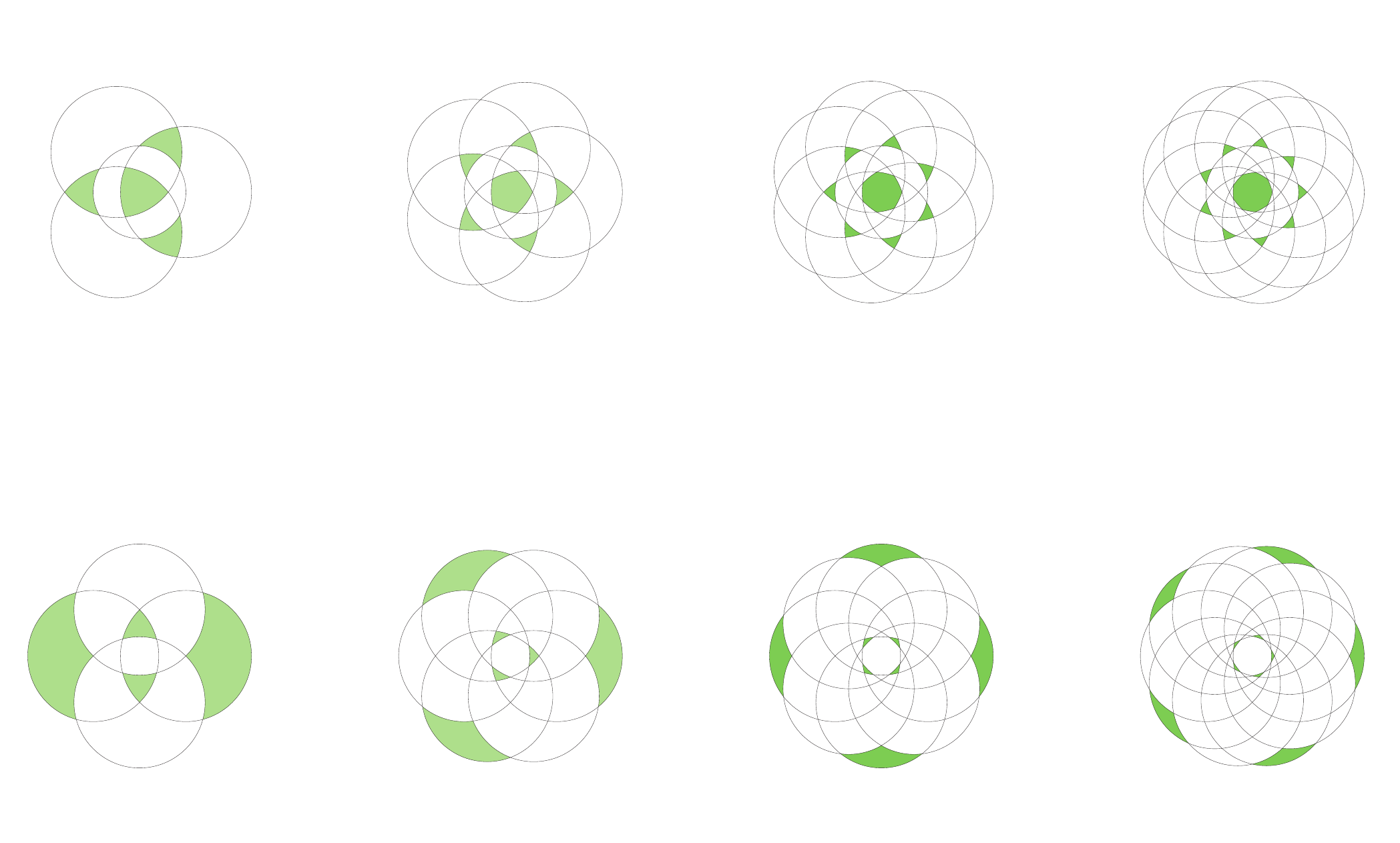} 
   \caption{Families A and B in stereographic projection}
   \label{fig4}
\end{figure}
\end{landscape}

%

\begin{figure}[htbp] 
\begin{minipage}[b]{0.45\linewidth}
\centering
\includegraphics[scale=0.3]{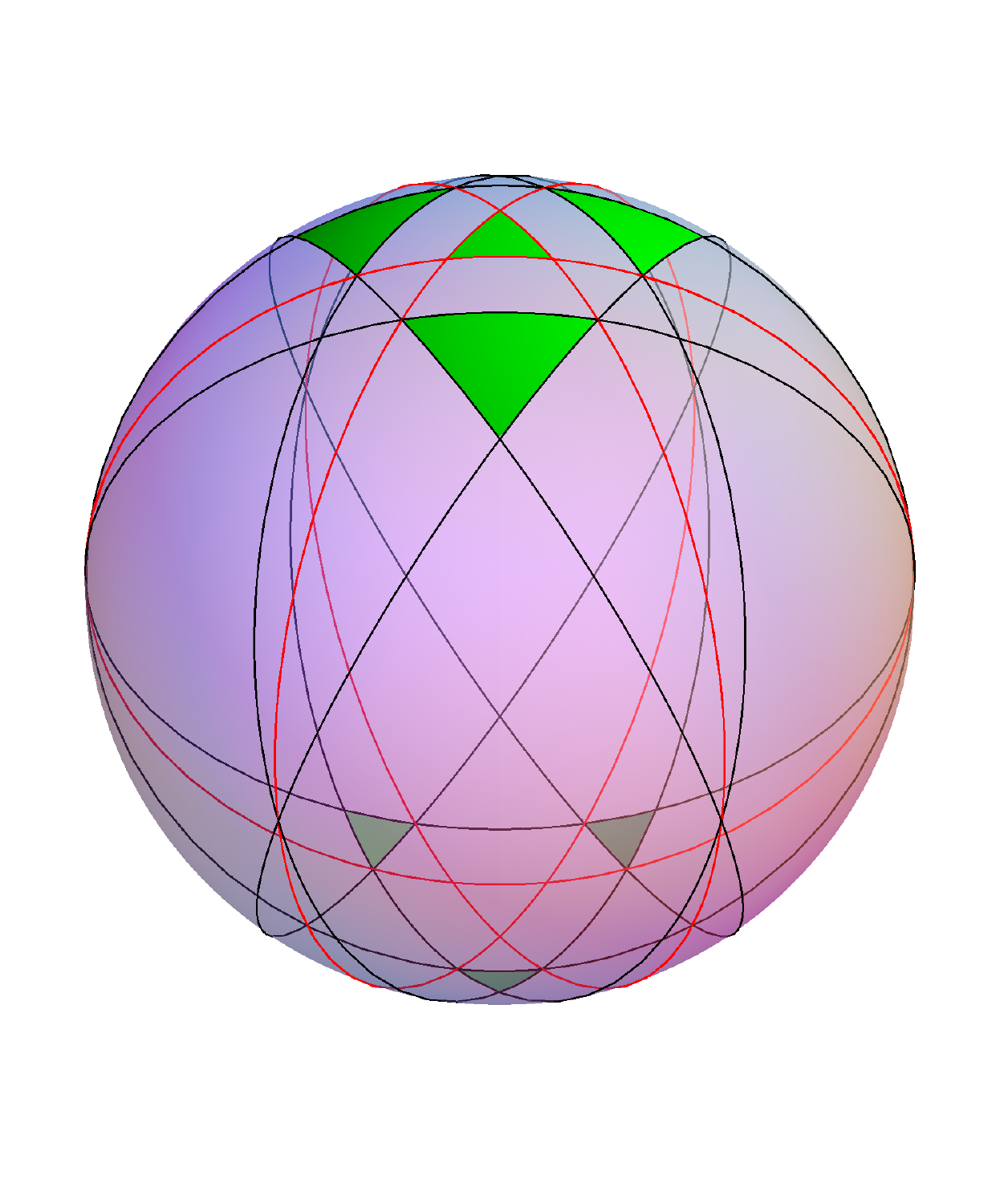}
\end{minipage}
\begin{minipage}[b]{0.45\linewidth}
\centering
\includegraphics[scale=0.3]{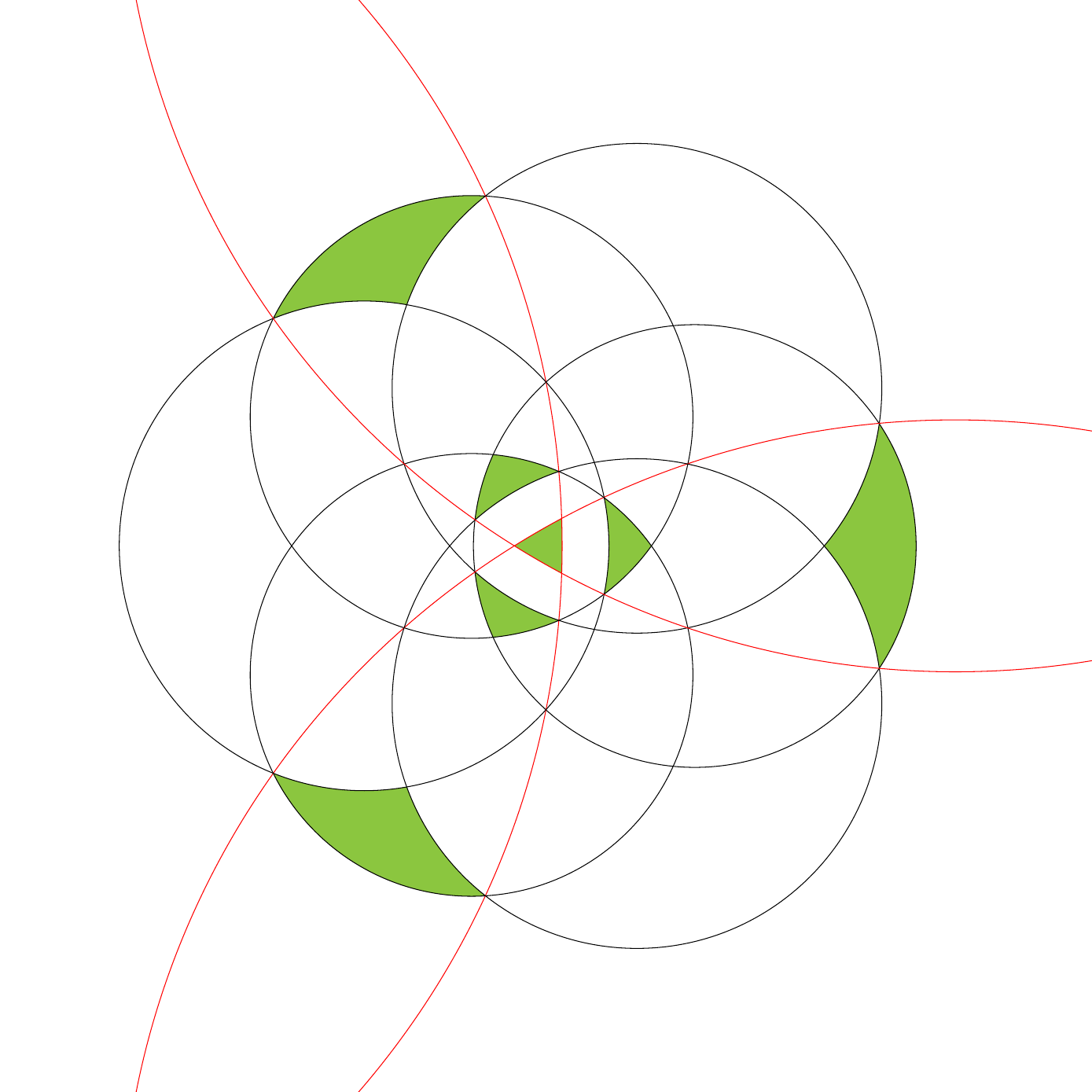}
\end{minipage}
   \caption{A maxtame set with $7$ components}
   \label{fig5fig6}
\end{figure}

\begin{lemma}
The configuration space of maxtame subsets of $S^{1}$ is homeomorphic to the infinite sphere $S^{\infty}$.
\end{lemma}
\begin{proof}
In view of the above remarks it suffices to show that the configuration space of odd numbers of points on $\RP^{1}$ is homeomorphic to $\RP^{\infty}$. Represent $\RP^{1}$ as $[0,1]/\{0\sim1\}$.

Let $\ell$ be an odd number.  Then the configuration space $C_\ell$ of $\ell$ points is a quotient of $[0,1]^{\ell}$, where we quotient by the 
action of the symmetric group $S_{\ell}$ and also by the equivalence relation
$(x_{1},\dots,x_{\ell})\sim(y_{1},\dots,y_{\ell})$ if and only if there exists indices $i\ne j$ such that $x_{i}=x_{j}$, $y_{i}=y_{j}$ and $x_{k}=y_{k}$ for $k\in\{1,\dots,\ell\}\setminus\{i,j\}$.

We almost get a fundamental domain for the $S_{n}$-action by taking 
\[\{(x_{1},\dots,x_{\ell});\ 0\le x_1 \le x_2 \le ... \le x_\ell \le 1\}.\]
This gives a polygonal subset 
$B$, homeomorphic to a ball (indeed it is an $\ell$-simplex), of the unit cube.  The face $x_1=0\sim1$ gets identified with the face $x_\ell=1$ by the $S_\ell$-action.  The other faces all have $x_i=x_{i+1}$ and thus get "crushed" into the boundary equator separating the above two faces.  Thus we have an $\ell$-ball with its boundary points identified in antipodal pairs, i.e. $\RP^\ell$. Note that the subset with $x_\ell=1$ is similarly homeomorphic to $\RP^{\ell-1}$ included into $\RP^\ell$ in the standard way.  Thus the subset with $x_{\ell-1}=x_\ell=1$ which is $C_{\ell-2}$ is $\RP^{\ell-2}$ included into $\RP^\ell$ in the standard way.  Thus the union of all $C_\ell$ is $\RP^\infty$.
\end{proof}

\section{Maximal open $2$-tame sets on $S^{1}$ with infinitely many components}

We include two examples showing that the finiteness condition $|\pi_{0}|<0$ cannot be removed from Definition 1.1 without severely altering the landscape. We'll view $S^{1}$ as the set of unit complex numbers.

Note that if $U$ is maximal open and $2$-tame on any sphere then $C:=(U\cup-U)^{c}$ is antipodally symmetric, closed and must have no interior points. In the first example we take the simplest of such sets, namely a symmetric pair of discrete sequences converging to limit points.

\begin{example}
Let $C$ denote the subset $\{\pm e^{\pi it};\ t=1,\frac12,\frac13,\dots,\frac1n,\dots\}\cup\{\pm1\}$ of $S^{1}$. Then there is a maximal open $2$-tame $U$ such that $C=(U\cup-U)^{c}$.
\end{example}
\begin{proof}
Let $f:[0,1]\to S^{1}$ be the map $t\mapsto e^{\pi it}$. Then we can take $U$ to be the union 
\[f(\tfrac12,1)\cup -f(\tfrac13,\tfrac12)\cup
f(\tfrac14,\tfrac13)\cup -f(\tfrac15,\tfrac14)\cup
f(\tfrac16,\tfrac15)\cup -f(\tfrac17,\tfrac16)\cup\cdots.\]
\end{proof}

Our second example involves a Cantor set. 

\begin{example}
We write all numbers in ternary instead of decimal. Thus $0.1\dot1$ denotes ``zero point one recurring'' and hence exactly a half. We consider the set $C$ of points in the interval $[0,11]$ (i.e. real numbers between zero and four) which can be expressed in ternary using only the digits $0$ and $2$. Note that recurring $2$s are permitted here even though a number such as $0.220\dot2$ is equal to the expression $0.221$ involving a $1$. 
Let $p:[0,11]\to S^{1}$ be defined by $t\mapsto e^{0.1\dot1\pi it}$.
Let $Y$ be the set of intervals of ternary numbers of the form
$(a_{0}\cdot a_{1}a_{2}\dots a_{n}1,a_{0}\cdot a_{1}a_{2}\dots a_{n}2)$ where the digits $a_{i}\in\{0,2\}$ include an even number of $2$s. Let $U$ be the set $\bigcup_{I\in Y}p(I)$. Then $U$ is maximal open and $2$-tame and  $U\cup-U$ has complement the Cantor set $p(C)$.
\end{example}
		
\section{Non-polygonal maximal open $2$-tame subsets of $S^{2}$}

The polyhedrality associated with maxtame sets does not apply in general to maximal open $k$-tame sets in $S^{n-1}$ when $k<n$ even if there are finitely many connected components. Figure \ref{fig7fig8}(left) shows a maximal open $2$-tame set obtained by starting with the open northern hemisphere, carving a flower shape, and placing the antipodal set in the southern hemisphere. Clearly any shape from the northern hemisphere can be used so long as it is the interior of its closure. Thus we see that polyhedrality fails for maximal open $2$-tame sets on $S^{2}$ and that such sets exists with any finite number of components $\ge1$. As well as the obvious example with a single component, namely a hemisphere, there are also exotic $1$-component maximal open $2$-tame sets such as the one illustrated in Figure \ref{fig7fig8}(right).

%

\begin{figure}[htbp] 
\begin{minipage}[b]{0.45\linewidth}
\centering
\includegraphics[scale=0.3]{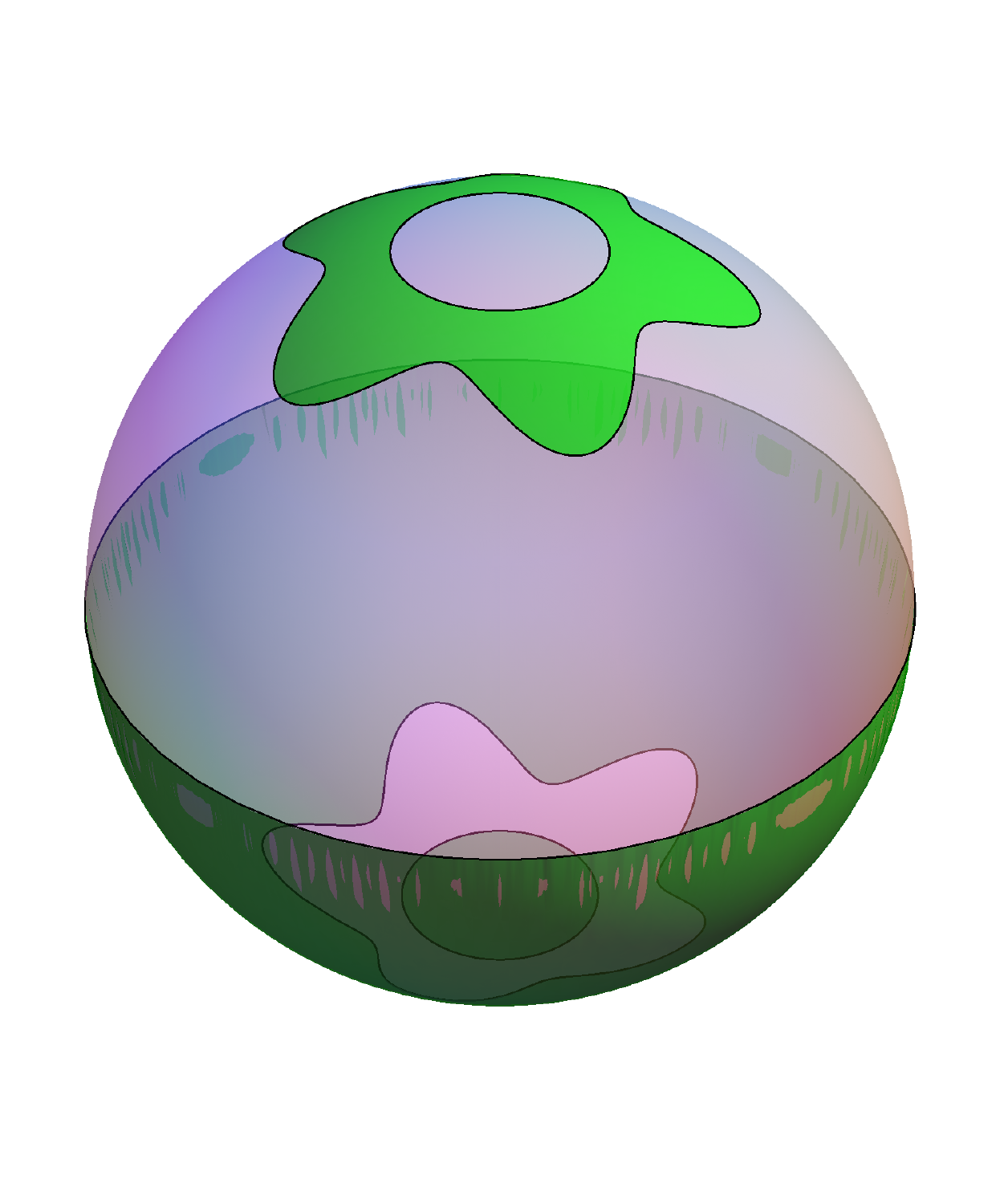}
\end{minipage}
\begin{minipage}[b]{0.45\linewidth}
\centering
\includegraphics[scale=0.35]{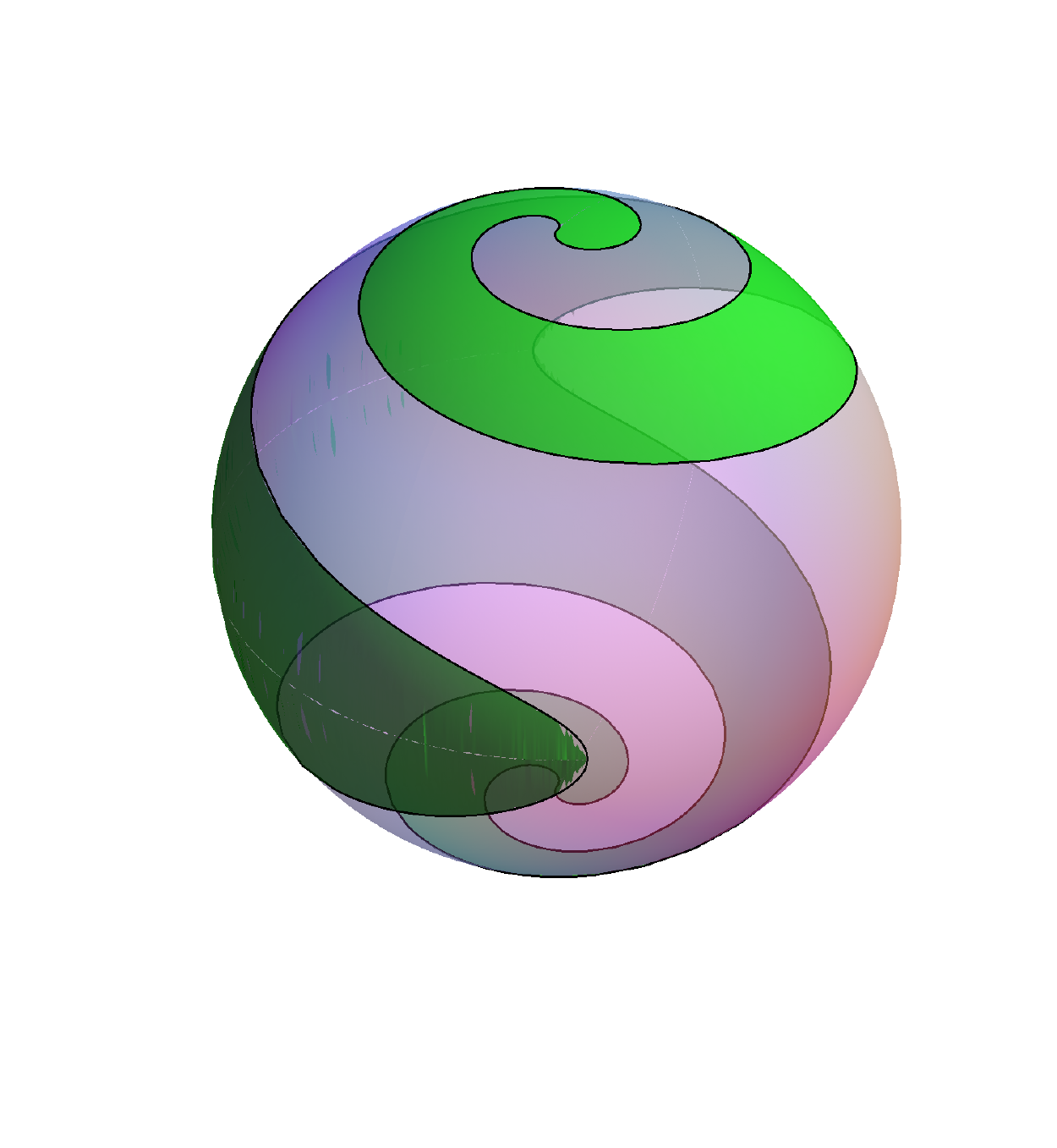}
\end{minipage}
   \caption{Non-polyhedral maximal open $2$-tame sets on $S^{2}$}
   \label{fig7fig8}
\end{figure}

\section{Maxtame subsets of $S^{n-1}$ with $n+1$ components}

In this section we explain how the $3$-component sets in Figure \ref{fig:2tameS1} and the $4$-component sets in Figures \ref{fig1},\ref{fig2},\ref{fig3},\ref{fig4} belong to a family in which the $n$th members are maxtame $(n+1)$-component subsets of $S^{n-1}$.

Fix $n\ge2$. Choose $n+1$ points $v_{0},\dots,v_{n}$ on the $(n-1)$-sphere such that their convex hull in $\R^{n}$ contains the origin as an interior point. For each $i$ choose a closed hemisphere $H_{i}$ which contains every $v_{j}$ with $j\ne i$ in its interior but not containing $v_{i}$. Let $\partial_{i}$ denote the boundary of $H_{i}$. These boundaries carve the sphere into $2^{n+1}-2$ regions in the following way. For each choice of $n+1$-tuple $(\epsilon_{0},\dots,\epsilon_{n})\in\{\pm\}^{n+1}$ the intersection $\bigcap_{i}\epsilon_{i}H_{i}$ is non-empty except in the two cases when all the $\epsilon_{i}$ are equal to each other. We use the shorthand notation $\epsilon_{0}\epsilon_{1}\dots\epsilon_{n}$ to denote the intersection. 

Let $U_{i}$ be the spherical $n$-simplex obtained by taking $\epsilon_{j}=+$ for $j\ne i$ and $\epsilon_{i}=-$. Let $U$ denote the union of the $U_{i}$. 

\begin{lemma}
The interior of $U$ is maxtame.
\end{lemma}
\begin{proof}
In order to explain the argument we restrict to the case $n=4$. The general case follows in the same way. Each $5$-tuple in $\{\pm\}^{5}$ which contains at least one plus and at least one minus determines a non-empty intersection of hemispheres. The five $5$-tuples $++++-$, $+++-+$, $++-++$, $+-+++$ and $-++++$ represent the five spherical tetrahedral components of $U$. Each of these has four boudary simplices. Replacing a $+$ or a $-$ by the symbol $\partial$ in the $i$th coordinate to indicate intersection with the boundary $\partial_{i}$ we see that the boundary of $++++-$ is made up of the four simplices $\partial+++-$, $+\partial++-$, $++\partial+-$, $+++\partial-$. By replacing one of the plus symbols by a minus, we pass across a boudary to an adjacent region. For example $++++-$ and $++-+-$ are adjacent with common boudary $++\partial+-$. The antipodal set $-U$ is made up of the five spherical tetrahedra $----+$,\dots,$+----$ described by the tuples with four minuses and one plus. 
Each of the components of $-U$ kisses every component of $U$ save the antipodal component, that is to say, when two of these components meet, they meet in a single point.
We also introduce the symbol $\pm$ into our tuples to indicate no restriction in the relevant coordinate. For example $++\pm--$ denotes the intersection $H_{0}\cap H_{1}\cap -H_{3}\cap -H_{4}$.

The fact that $U$ is $n$-tame is immediate from the construction. To check that it is maxtame we appeal to Lemma \ref{whatisaddable}.  We need to know how to compute the convex hull of the union of any subset of the five antipodal simplices $\{----+,\dots,+----\}$. The recipe is very simple. The convex hull consists of all regions $\epsilon_{0}\dots\epsilon_{n}$ in which $\epsilon_{i}$ takes one of the values occurring amongst the chosen simplices. In effect this removes any restriction from the coordinates where there are sign changes among the different chosen simplices. For example, the convex hull of 
$$+----\cup--+--$$ is $\pm-\pm--$ and the convex hull of 
$$+----\cup--+--\cup---+-$$ is 
$\pm-\pm\pm-$. We illustrate this with the proof that
$$\hull(+----\cup--+--\cup---+-)=\pm-\pm\pm-.$$
The inclusion
$$\hull(+----\cup--+--\cup---+-)\subseteq\pm-\pm\pm-$$
is immediate because each of $+----$, $--+--$, $---+-$ is a subset of 
$\pm-\pm\pm-$ and $\pm-\pm\pm-$ is convex. Now observe that $\pm-\pm\pm-$ is an intersection of two hemispheres and in this case, on $S^{3}$, this intersection is a lens bounded by hemi-$2$-spheres. The lens sits in the circular boundary $\pm\partial\pm\pm\partial$ and each point in the lens resides on a unique great hemi-$2$-sphere with boundary this circle. Each of the three tetrahedra $+----$, $--+--$, $---+-$ shares one of its edges with a segment on the circle and in this way, the convex hull of the circle tetrahedra contains the circle and therefore all points of the lens.

It follows that the maxtame condition is satisfied: the interior of $U$ is equal $-U[4]$.
\end{proof}

\section{Proof that a maxtame subset of $S^{n-1}$ cannot have $n+2$ components}

On $S^{2}$ we have seen no examples of sets with $2$, $3$, or $5$ components. In the case of $2$ and $3$ the reason is a very easy consequence of Lemma \ref{keylemma}.

\begin{proposition}
Every maxtame subset of $S^{n-1}$ with at most $n$ components is a hemisphere.
\end{proposition}
\begin{proof}
Let $U$ be a maxtame set with $|\pi_{0}(U)|\le n$.
Let $X$ be a finite set with exactly one point in each component. Then $X$ is contained in an open hemisphere by $n$-tameness of $U$ and therefore  the Key Lemma \ref{keylemma} shows that there is a hemisphere $H$ containing $X$ whose boundary does not meet $U$. Connectivity of each component guarantees that $U$ is contained in the interior of $H$. Since $U$ is maxtame it follows that $U$ equals the interior of $H$.
\end{proof}

That there are no $5$-component maxtame sets on $S^{2}$ also generalizes to higher dimensons.

\begin{proposition}
Let $U$ be an open $n$-tame subset of $S^{n-1}$ with exactly $n+2$ components. Then $U$ is contained in a weakly maxtame set with at most $n+1$ components.
\end{proposition}
\begin{proof}
Choose representative points $x_{1},\dots,x_{n+2}$, taking one from each component of $U$. Let $U_{i}$ be the component containing $x_{i}$. 
Set $X:=\{x_{1},\dots,x_{n+2}\}$.
By $n$-tameness, the simplex spanned by any $n$-point subset of $X$ is contained in an open hemisphere. In particular, the convex hulls in $\R^{n}$ of the $n$-point subsets do not contain the origin. If the convex hull of $X$ does not contain the origin then the Key Lemma \ref{keylemma} shows that $X$ is contained in an open hemisphere which does not meet $U$ and connectivity shows that $U$ itself is contained in this open hemisphere. So we may as well assume that the convex hull of $X$ does contain the origin. In that case, some $n+1$ points of $X$ have a convex hull which contain the origin and without loss of generality we assume that the simplex $\Delta$ spanned by the set $\{x_{1},\dots,x_{n+1}\}$ contains the origin within its interior. Imagine viewing $\Delta$ from the standpoint of $x_{n+2}$. We may assume that exactly $i$ of the $(n-1)$-dimensional faces of $\Delta$ are visible from $x_{n+2}$ and exactly $n+1-i$ of these faces are not visible. Coning each of the $i$ visible faces of $\Delta$ off to $x_{n+2}$ produces an $n$-simplex which does not contain the origin. Coning each of the $n-i$ non-visible faces off to $x_{n+2}$ produces an $n$-simplices which do not overlap with each other and therefore at most one of these can contain the origin.

We may conclude from this first part of the argument that the simplex in $\R^n$ spanned by $n+1$ points chosen from $X$ contains the origin in at most $2$ cases. Let $\Delta_{i}$ denote the simplex spanned by $X\setminus\{x_{i}\}$, (so $\Delta=\Delta_{n+2}$). Without loss of generality we may assume that $\Delta_{1},\dots,\Delta_{n}$ do not contain the origin and that $\Delta_{n+1}$ and $\Delta_{n+2}$ both do contain the origin. Using the Key Lemma \ref{keylemma} we conclude that there are open hemispheres $H_{1},\dots,H_{n}$ whose boundaries do not meet $U$ and such that $X\setminus \{x_{i}\}\subset H_{i}$, and that there is a further open hemisphere $H$, again with boundary not meeting $U$, so that 
$X\setminus\{x_{n+1},x_{n+2}\}\subset H$. Connectivity and convexity now imply that we may replace $U_{n+1}\cup U_{n+2}$ by its convex hull and so enlarge $U$ without violating its $n$-tameness. This procedure reduces the number of components of $U$ and the result follows from Theorem B1 together with the note following the proof of Lemma \ref{sec2:lemma3}.
\end{proof}

Thus we have the immediate

\begin{corollary}
Every maxtame subset of $S^{n-1}$ has $1$, $n+1$, or $\ge n+3$ components.
\end{corollary}

\bibliography{peter}
\bibliographystyle{abbrv}

\end{document}